\documentclass{amsart}
\usepackage{amsmath}
  \usepackage{paralist}
  \usepackage{graphics} 
  \usepackage{epsfig} 
\usepackage{graphicx}  \usepackage{epstopdf}
 \usepackage[colorlinks=true]{hyperref}
\hypersetup{urlcolor=blue, citecolor=red}
\usepackage{verbatim,dsfont,amssymb,lipsum,subcaption}

\newtheorem{theorem}{Theorem}[section]
\newtheorem{corollary}[theorem]{Corollary}

\newtheorem{lemma}[theorem]{Lemma}
\newtheorem{proposition}[theorem]{Proposition}

\theoremstyle{definition}
\newtheorem{definition}[theorem]{Definition}
\newtheorem{remark}[theorem]{Remark}

\newtheorem{assumptions}[theorem]{Assumptions}
\newtheorem{example}[theorem]{Example}
\newtheorem{algorithm}[theorem]{Algorithm}

\newcommand{\eps}{\varepsilon}

\newcommand{\dee}{\mathrm{d}}

\newcommand{\mmd}{}

\graphicspath{{./images/}}

\title[The Bayesian Formulation of EIT] 
      {The Bayesian Formulation of EIT:\\Analysis and Algorithms}

\author[M. M. Dunlop and A. M. Stuart]{}


 \email{mdunlop@caltech.edu}
 \email{astuart@caltech.edu}

\begin{document}
\maketitle

\centerline{\scshape Matthew M. Dunlop}
\medskip
{\footnotesize
 \centerline{Computing \& Mathematical Sciences,}
   \centerline{California Institute of Technology,}
   \centerline{Pasadena, CA 91125, USA}
} 

\medskip

\centerline{\scshape Andrew M. Stuart}
\medskip
{\footnotesize
 \centerline{Computing \& Mathematical Sciences,}
   \centerline{California Institute of Technology,}
   \centerline{Pasadena, CA 91125, USA}
}


\begin{abstract}
We provide a rigorous Bayesian formulation of the EIT problem in an infinite dimensional setting, leading to well-posedness in the Hellinger metric with respect to the data. We focus particularly on the reconstruction of binary fields where the interface between different media is the primary unknown. We consider three different prior models -- log-Gaussian, star-shaped and level set. Numerical simulations based on the implementation of MCMC are performed, illustrating the advantages and disadvantages of each type of prior in the reconstruction, in the
case where the true conductivity is a binary field, and exhibiting the properties of the resulting posterior distribution.
\end{abstract}

\section{Introduction}
\subsection{Background}

Electrical Impedance Tomography (EIT) is an imaging technique in which the conductivity of a body is inferred from electrode measurements on its surface. Examples include medical imaging, where the methodology is used to non-invasively detect abnormal tissue within a patient, and subsurface imaging where material properties of the subsurface are determined from surface (or occasional interior) measurements of the electrical response; the methodology is  often referred to as electrical resistance tomography -- ERT -- in this context and discussed in more detail below. The concept of EIT appears as early as the late 1970's \cite{eit_early} and ERT the 1930's \cite{ert}.
 
A very influential  mathematical formulation of the inverse problem associated with EIT dates back to 1980, due to Calder\'on. He formulated an abstract version of the problem, in which the objective is recovery of the coefficient of a divergence form elliptic PDE from knowledge of its Neumann-to-Dirichlet or Dirichlet-to-Neumann operator. Specifically, in the Dirichlet-to-Neumann case, if $D\subseteq\mathbb{R}^d$ and $g \in H^{1/2}(\partial D)$ is given, let $u \in H^1(D)$ solve
\[
\begin{cases}
-\nabla\cdot(\sigma\nabla u) = 0 &\text{ in }D\\
\hspace{1.6cm}u = g &\text{ on }\partial D.
\end{cases}
\]
The problem of interest is then to ask does the mapping $\Lambda_\sigma:H^{1/2}(\partial D)\rightarrow$\linebreak$H^{-1/2}(\partial D)$ given by
\[
g\mapsto \sigma\frac{\partial u}{\partial \nu}
\]
determine the coefficient $\sigma$ \mmd{in $D$}? Physically, $g$ corresponds to boundary voltage measurements, and $\Lambda_\sigma(g)$ corresponds to the current density on the boundary. Much study has been carried out on this problem -- some significant results, in the case where all conductivities are in $C^2(\overline{D})$ and $d \geq 3$, concern uniqueness \cite{uniqueness}, reconstruction \cite{reconstruction}, stability \cite{stability} and partial data \cite{partialdata}. Details of these results are summarised in \cite{calderon}.

In 1996, Nachman proved global uniqueness and provided a reconstruction procedure for the case $d = 2$, involving the use of a scattering transform and solving a D-bar problem \cite{nachman96}. The D-bar equation involved is a differential equation of the form $\overline{\partial}q = f$, where $\overline{\partial}$ denotes the conjugate of the complex derivative and $f$ depends on the scattering transform. A regularised D-bar approach, involving the truncation of the scattering transform, was provided in \cite{eit_disc2,eit_disc1}, enabling the recovery of features of discontinuous permeabilities. The regularised D-bar approach is also used in \cite{regdbar}, for the case when the data is not of infinite precision. Other work in the area includes joint inference of the shape of the domain and conductivity \cite{eit_shape}.

For applications, a more physically appropriate model for EIT was provided in 1992 in \cite{cheney}. This model, referred to as the complete electrode model (CEM), replaces complete boundary potential measurements with measurements of constant potential along electrodes on the boundary, subject to contact impedances. The authors show that predictions from this model agree with experimental measurements up to the measurement precision of 0.1\%. For this model they also prove existence and uniqueness of the associated electric potential. It is this model that we shall consider in this paper, and it is outlined in section \ref{sec:forward}. 

When using the CEM, there is a limitation on the number of measurements that can be taken to provide additional information. \mmd{This is due to the fact that there are only a finite number of electrodes through which current can be injected and voltages read, combined with the} linear relationship between current and voltage. The data is therefore finite dimensional in the inverse problem, as distinct from the Calder\'on problem where knowledge of an infinite dimensional operator is assumed. As a consequence, reconstruction using the CEM often makes use of Tikhonov regularisation; \mmd{such a regularisation can also be used to account for noise on the data}. The paper \cite{eit_fem} analyses numerical convergence when an $H^1$ or TV penalty term is used, with a finite element discretisation of the problem. We will adopt a Bayesian approach to regularisation, and this is discussed below.

A closely related problem to EIT is Electrical Resistivity Tomography (ERT), which concerns subsurface inference from surface potential measurements, see for example \cite{ert} which discussed the problem as early as 1933. Physically the main difference between EIT and ERT is that alternating current is typically used for the former, and direct current for the latter. Additionally, due to the scale of ERT, it is a reasonable approximation to model the electrodes as points, rather than using the CEM. Another difference between the two is that the relative contrast between the conductivities of different media are typically higher in subsurface applications than medical applications, which permits the approximation of the problem by a network of resistors in some cases \cite{network}. Nonetheless, the Bayesian theory and MCMC methodology introduced here will be useful for the ERT problem as well as the EIT problem.

Statistical, in particular Bayesian, approaches to EIT inversion have previously been studied, for example in \cite{frechet,bayes1,whittlematern}. In \cite{frechet}, the authors prove certain regularity of the forward map associated with the CEM, formulate the Bayesian inverse problem in terms of the discretised model, and investigate the effect of different priors on reconstruction and behaviour of the posterior. The paper \cite{whittlematern} focuses on Whittle-Mat\`ern priors, using EIT and ERT as examples for numerical simulation. The paper \cite{bayes1} presents a regularised version of the inverse problem, which admits a Bayesian interpretation.

The Bayesian approach to inverse problems, especially in infinite dimensions, is a relatively new technique. Two approaches are typically taken: discretise first and use finite dimensional Bayesian techniques, or apply the Bayesian methodology directly on function space \mmd{before} discretising. The former approach is outlined in \cite{kaipio}. The latter approach for linear problems has been studied in \cite{linear1,lassas2009discretization,linear3,linear2}. More recently, this approach has been applied to nonlinear problems \cite{lecturenotes,lasanen1,lasanen2,inverse}. It is this approach that we will be taking in this paper.

\subsection{Our Contribution}
The main contributions in our paper are as follows:
\begin{enumerate}[(i)]
\item This is the first paper to give a rigorous Bayesian formulation of EIT in infinite dimensions.
\item This setting leads to proof that the posterior measure is Lipschitz in the data, with respect to the Hellinger metric, for all three priors studied; further stability properties of the posterior with respect to perturbations, such as numerical approximation of the forward model, may be proved similarly.
\item We employ a variety of prior models, based on the assumption that the underlying conductivity we wish to recover is binary. We initially look at log-Gaussian priors, before focusing on priors which enforce the binary field property. These binary field priors include both single star-shaped inclusions parametrised by their centre and by a radial function \cite{star}, and arbitrary geometric interfaces between the two conductivity values defined via level set functions \cite{levelset}.
\item Numerical results using state of the art MCMC demonstrate the importance of the prior choice for accurate reconstruction in the severely underdetermined inverse problems arising in EIT.
\end{enumerate}
\subsection{Organisation of the Paper}
In section \ref{sec:forward} we describe the forward map associated with the EIT problem, and prove relevant regularity properties. In section \ref{sec:post} we formulate the inverse problem rigorously and describe our three prior models. We then prove existence and well-posedness of the posterior distribution for each of these choices of prior.  In section \ref{sec:num} we present results of numerical MCMC simulations to investigate the effect of the choice of prior on the recovery of certain binary conductivity fields. We conclude in section \ref{sec:conc}.

\section{The Forward Model}
\label{sec:forward}
In subsection \ref{ssec:cem} we describe the complete electrode model for EIT as given in \cite{cheney}. In subsection \ref{ssec:wcem} we give the weak formulation of this model, stating assumptions required for the quoted existence and uniqueness result. Then in subsection \ref{ssec:fwd} we define the forward map in terms of this model, and prove that this map is continuous with respect to both uniform convergence and convergence in measure.
\subsection{Problem Statement}
\label{ssec:cem}
Let $D\subseteq \mathbb{R}^d$, $d \leq 3$, be a bounded open set \mmd{with smooth boundary} representing a body, with conductivity  $\sigma:\overline{D}\rightarrow\mathbb{R}$. A number $L$ of electrodes are attached to the surface of the body. We treat these as subsets $(e_l)_{l=1}^L$ of the boundary $\partial D$, and assume that they have contact impedances $(z_l)_{l=1}^L \in \mathbb{R}^L$. A current stimulation pattern $(I_l)_{l=1}^L \in \mathbb{R}^L$ is applied to the electrodes. Then the electric potential $v$ within the body and boundary voltages $(V_l)_{l=1}^L \in \mathbb{R}^L$ on $(e_l)_{l=1}^L$ are modelled by the following partial differential equation. 

\begin{align}
\label{eq:strongpde}
\begin{cases}
\displaystyle-\nabla\cdot(\sigma(x)\nabla v(x)) = 0 & x \in D\\[.8em]
\displaystyle\int_{e_l} \sigma\frac{\partial v}{\partial n}\,\dee S = I_l & l=1,\ldots,L\\[.8em]
\displaystyle\sigma(x)\frac{\partial v}{\partial n}(x) = 0 & x \in \partial D\setminus\bigcup_{l=1}^L e_l\\[.8em]
\displaystyle v(x) + z_l\sigma(x)\frac{\partial v}{\partial n}(x) = V_l & x \in e_l, l=1,\ldots,L 
\end{cases}
\end{align}

This model was first proposed in \cite{cheney}; a derivation can be found therein. Note that the inputs to this forward model are the conductivity $\sigma$, input current $(I_l)_{l=1}^L$ and contact impedances $(z_l)_{l=1}^L$. The solution comprises the function $v:D\rightarrow\mathbb{R}$ \emph{and} the vector $V \in \mathbb{R}^L$ of voltages. Also note that solutions to this equation are only defined up to addition of a constant: if $(v,V)$ solves the equation, then so does $(v+c,V+c)$ for any $c \in \mathbb{R}$. This is because it is necessary to choose a reference ground voltage.

\begin{figure}[h!]
\begin{center}
\includegraphics[width=0.5\textwidth, trim= 1cm 1cm 1cm 1cm]{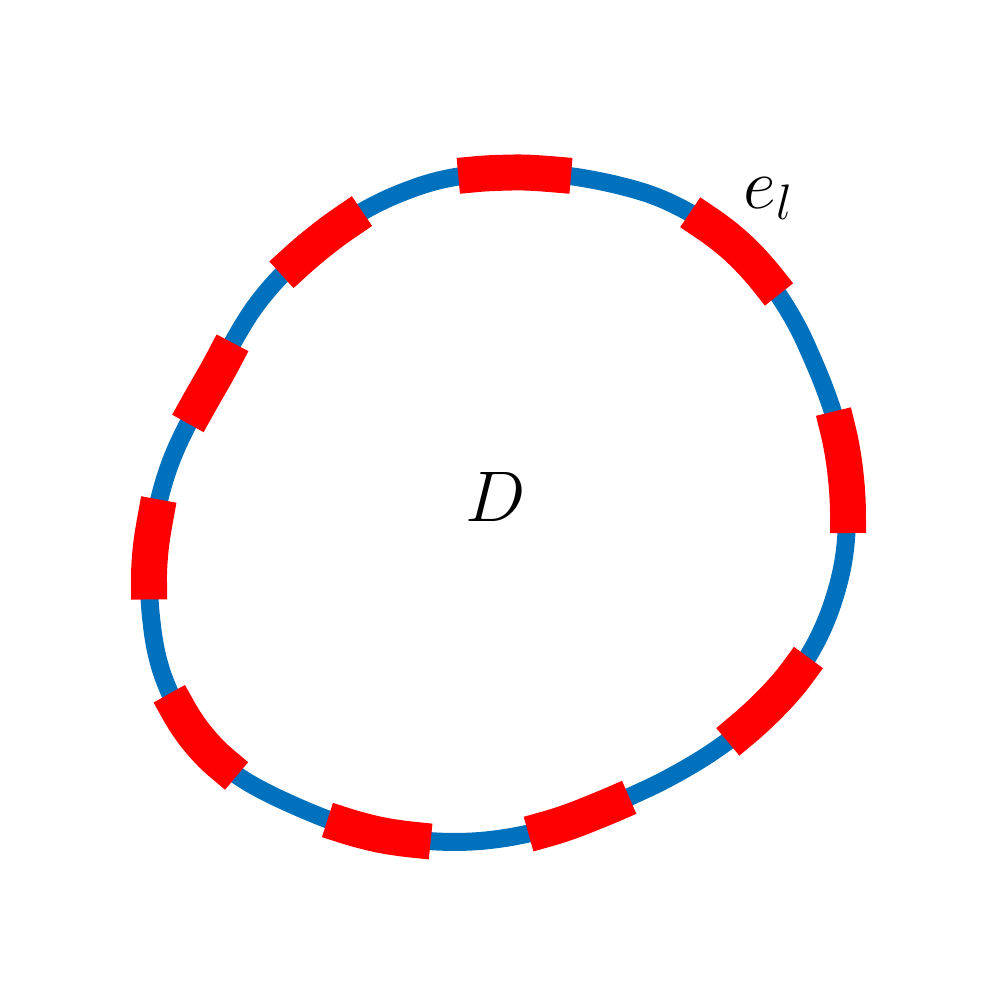}
\end{center}
\caption{An example domain $D$ with electrodes $(e_l)_{l=1}^L$ attached to its boundary.}
\end{figure}

\subsection{Weak Formulation}
\label{ssec:wcem}
We first define the space in which the solution to equation (\ref{eq:strongpde}) will live. Using the notation of \cite{cheney} we set 
\begin{align*}
\mathbb{H} &= H^1(D)\oplus\mathbb{R}^L,\\
 \|(v,V)\|_{\mathbb{H}}^2 &= \|v\|_{H^1(D)}^2 + \|V\|_{\ell^2}^2\\
 &= \|v\|_{L^2(D)}^2 + \|\nabla v\|_{L^2(D)}^2 + \|V\|_{\ell^2}^2.
\end{align*}
Since solutions are only defined up to addition of a constant, we define the quotient space $(\dot{\mathbb{H}},\|\cdot\|_{\dot{\mathbb{H}}})$ by
\begin{align*}
\dot{\mathbb{H}} &= \mathbb{H}/\mathbb{R},\\
\|(v,V)\|_{\dot{\mathbb{H}}} &= \inf_{c \in \mathbb{R}} \|(v - c, V - c)\|_{\mathbb{H}}.
\end{align*}
We will often use the notation $v' = (v,V)$ for brevity. It is more convenient to equip $\dot{\mathbb{H}}$ with an equivalent norm, as stated in the following lemma from \cite{cheney}:
\begin{lemma}
\label{lem:normequiv}
Define $\|v'\|_*$ by
\[
\|v'\|_*^2 = \|\nabla v\|_{L^2(D)}^2 + \sum_{l=1}^L\int_{e_l} |v(x)-V_l|^2\,\dee S.
\]
Then $\|\cdot\|_*$ and $\|\cdot\|_{\dot{\mathbb{H}}}$ are equivalent.
\end{lemma}

We can now state the weak formulation of the problem as derived in \cite{cheney}. For this let $w' = (w,W)$.
\begin{proposition}
\label{lem:weakpde}
Let $B:\dot{\mathbb{H}}\times\dot{\mathbb{H}}\rightarrow\mathbb{R}$ and $r:\dot{\mathbb{H}}\rightarrow\mathbb{R}$ be defined by
\begin{align*}
B(v',w';\sigma) &= \int_D \sigma\nabla v\cdot\nabla w\,\dee x + \sum_{l=1}^L\frac{1}{z_l}\int_{e_l} (v-V_l)(w-W_l)\,\dee S,\\
r(w') &= \sum_{l=1}^L I_l W_l.
\end{align*}
Then if $v' \in \dot{\mathbb{H}}$ is a strong solution to the problem (\ref{eq:strongpde}), it satisfies
\begin{align}
\label{wpde}
B(v',w';\sigma) = r(w')\;\;\;\text{for all }w' \in \dot{\mathbb{H}}.
\end{align}
\end{proposition}

We will use the weak formulation (\ref{wpde}) to define our forward map for the complete electrode model (\ref{eq:strongpde}). In order to guarantee a solution to this problem, we make the following assumptions.
\begin{assumptions}
\label{assump:exist}
The conductivity $\sigma:\overline{D}\rightarrow\mathbb{R}$, the contact impedances $(z_l)_{l=1}^L \in \mathbb{R}^L$ and the current stimulation pattern $(I_l)_{l=1}^L \in \mathbb{R}^L$ satisfy
\begin{enumerate}[(i)]
\item $\displaystyle\sigma \in L^{\infty}(\overline{D};\mathbb{R}),\;\;\;\underset{x \in D}{\mathrm{essinf}}\;\sigma(x) = \sigma_- > 0$;
\item $\displaystyle 0 < z_- \leq z_l \leq z_+ < \infty,\;\;\;l=1,\ldots,L$;
\item $\displaystyle \sum_{l=1}^L I_l = 0$.
\end{enumerate}
\end{assumptions}
Under these assumptions, existence of a unique solution to (\ref{wpde}) is proved in \cite{cheney} and stated here for convenience:
\begin{proposition}
Let Assumptions \ref{assump:exist} hold, then there is a unique $[(v,V)] \in \dot{\mathbb{H}}$ solving (\ref{wpde}). We may, without loss of generality, choose the element $(v,V) \in [(v,V)]$ of the equivalence class of solutions to be that which satisfies
\begin{align}
\label{groundvoltage}
\sum_{l=1}^L V_l = 0.
\end{align}
\end{proposition}

\begin{remark}
Assumptions \ref{assump:exist} (i) and (ii) are to ensure coercivity and boundedness of $B(\cdot,\cdot;\sigma)$. Assumption \ref{assump:exist} (iii) is necessary for continuity of  $r(\cdot)$, and physically may be thought of as a conservation of charge condition. Choosing a solution from the equivalence class corresponds to choosing a reference ground voltage.
\end{remark}

\subsection{Continuity of the Forward Map}
\label{ssec:fwd}
In what follows we will restrict to the set of admissible conductivities, which is defined as follows.
\begin{definition}
A conductivity field $\sigma:\overline{D}\rightarrow\mathbb{R}$ is said to be admissible if
\begin{enumerate}[(i)]
\item there exists $N \in \mathbb{N}$, $\{D_n\}_{n=1}^N$ open disjoint subsets of $D$ for which $\overline{D} = \bigcup_{n=1}^N \overline{D}_j$;
\item $\sigma\big|_{D_j} \in C(\overline{D}_j)$; and
\item there exist $\sigma^-, \sigma^+ \in (0,\infty)$ such that $\sigma^- \leq \sigma(x) \leq \sigma^+$ for all $x \in \overline{D}$.
\end{enumerate}
The set of all such conductivities will be denoted $\mathcal{A}(D)$.
\end{definition}

Note that any $\sigma \in \mathcal{A}(D)$ will satisfy Assumptions \ref{assump:exist}(i). Assume that the current stimulation pattern $(I_l)_{l=1}^L \in \mathbb{R}^L$ and contact impedances $(z_l)_{l=1}^L \in \mathbb{R}^L$ are known and satisfy Assumptions \ref{assump:exist}(ii)-(iii). Then we may define the solution map $\mathcal{M}:\mathcal{A}(D)\rightarrow\mathbb{H}$ to be the unique solution to (\ref{wpde}) satisfying (\ref{groundvoltage}). The above existence and uniqueness result tells us that this map is well-defined.

In \cite{frechet} it is shown that $\mathcal{M}:\mathcal{A}(D)\rightarrow\mathbb{H}$ is Fr\'echet differentiable when we equip $\mathcal{A}(D)$ with the supremum norm. Though this is a strong result, this choice of norm is not appropriate for all of the conductivities that we will be considering. We hence establish the following continuity result.

\begin{proposition}
\label{prop:fwdcont}
Fix a current stimulation pattern $(I_l)_{l=1}^L \in \mathbb{R}^L$ and contact impedances $(z_l)_{l=1}^L \in \mathbb{R}^L$ satisfying Assumptions \ref{assump:exist}. Define the solution map $\mathcal{M}:\mathcal{A}(D)\rightarrow\mathbb{H}$ as above. Let $\sigma \in \mathcal{A}(D)$ and let $(\sigma_\eps)_{\eps > 0} \subseteq \mathcal{A}(D)$ be such that either
\begin{enumerate}[(i)]
\item $\sigma_\eps$ converges to $\sigma$ uniformly; or
\item $\sigma_\eps$ converges to $\sigma$ in \mmd{(Lebesgue)} measure, and there exist $\sigma^-,\sigma^+ \in (0,\infty)$ such that for all $\eps > 0$ and $x \in D$, $\sigma^- \leq \sigma_\eps(x) \leq \sigma^+$.
\end{enumerate}
Then $\|\mathcal{M}(\sigma_\eps) - \mathcal{M}(\sigma)\|_{*}\rightarrow 0$.
\end{proposition}

\begin{proof}
Define the maps $B:\mathbb{H}\times\mathbb{H}\times\mathcal{A}(D)\rightarrow\mathbb{R}$ and $r:\mathbb{H}\rightarrow\mathbb{R}$ as in Lemma \ref{lem:weakpde}, but on $\mathbb{H}$ rather than $\dot{\mathbb{H}}$. Then denoting $u_\eps' = (v_\eps,V^\eps) = \mathcal{M}(\sigma_\eps)$ and $v' = (v,V) = \mathcal{M}(\sigma)$, we have for all $w' \in \mathbb{H}$,
\[
B(v_\eps',w';\sigma_\eps) = r(w'),\;\;\;B(v',w';\sigma) = r(w').
\]
It follows that
\begin{align*}
0 &= B(v_\eps',w';\sigma_\eps) - B(v',w';\sigma_\eps) + B(v',w';\sigma_\eps) - B(v',w';\sigma)\\
&= \int_D \sigma_\eps\nabla(v_\eps - v)\cdot\nabla w\,\dee x + \sum_{l=1}^L \frac{1}{z_l}\int_{e_l}((v_\eps - v) - (V^\eps_l - V_l))(w-W_l)\,\dee S\\
&\hspace{1cm}+ \int_D (\sigma_\eps - \sigma)\nabla v\cdot \nabla w\,\dee x.
\end{align*}
Letting $w' = (v_\eps-v,V^\eps - V)$, we see that
\begin{align*}
\int_D \sigma_\eps |\nabla (v_\eps - v)|^2\,\dee x + \sum_{l=1}^L \frac{1}{z_l}\int_{e_l}((v_\eps - v) &- (V^\eps_l - V_l))^2\,\dee S\\
&\leq \int_D |\sigma_\eps - \sigma||\nabla v\cdot \nabla (v_\eps - v)|\,\dee x.
\end{align*}
In both cases (i) and (ii), we have that $(\sigma_\eps)_{\eps>0}$ is bounded uniformly below by a positive constant. Hence for small enough $\eps$, the left hand side above may be bounded below by $C\|v_\eps' - v'\|_*^2$. We then have by Cauchy-Schwarz
\begin{align}
\|v_\eps'-v'\|_*^2 &\leq C\int_D |\sigma_\eps - \sigma||\nabla v\cdot \nabla (v_\eps - v)|\,\dee x\notag\\
&\leq C\left(\int_D |\sigma_\eps - \sigma|^2|\nabla v|^2\,\dee x\right)^{1/2}\cdot \|\nabla(v_\eps - v)\|_{L^2}\notag\\
&\label{eq:propcts1}\leq C\left(\int_D |\sigma_\eps - \sigma|^2|\nabla v|^2\,\dee x\right)^{1/2}\cdot \|v_\eps' - v'\|_{*}\\
&\label{eq:propcts2}\leq C\|\sigma_\eps-\sigma\|_\infty\|\nabla v\|_{L^2}\|v_\eps' - v'\|_{*}.
\end{align}
If $\sigma_\eps\rightarrow\sigma$ uniformly, we deduce from (\ref{eq:propcts2}) that $\|v_\eps'-v'\|_*\rightarrow 0$ and the result follows. If $|\sigma_\eps - \sigma|\rightarrow 0$ in measure, then since $|D|<\infty$, it follows that the integrand in (\ref{eq:propcts1}) tends to zero in measure, see for example Corollary 2.2.6 in \cite{bogachev1}. Since $\sigma_\eps$ is assumed to be uniformly bounded, the integrand is dominated by a scalar multiple of the integrable function $|\nabla v|^2$. We claim that this implies that the integrand tends to zero in $L^1$. Suppose not, and denote the integrand $f_\eps$. Then there exists $\delta > 0$ and a subsequence $(f_{\eps_i})_{i\geq 1}$ such that $\|f_{\eps_i}\|_{L^1} \geq \delta$ for all $i$. This subsequence still converges to zero in measure, and so admits a further subsequence that converges to zero almost surely. An application of the dominated convergence theorem leads to a contradiction, hence we deduce that $f_\eps$ tends to zero in $L^1$ and the result follows.
\end{proof}

Denote the projection $\Pi:\mathbb{H}\rightarrow\mathbb{R}^L$, $(v,V)\mapsto V$. The following lemma shows that the above result still holds if we replace $\mathcal{M}$ by $\Pi\circ\mathcal{M}$.

\begin{corollary}
\label{cor:fwdcont}
Let the assumptions of Proposition \ref{prop:fwdcont} hold. Then
\[
|\Pi\circ\mathcal{M}(\sigma_\eps) - \Pi\circ\mathcal{M}(\sigma)|_{\ell^2}\rightarrow 0.
\]
\end{corollary}
\begin{proof}
We show that there exists $C > 0$ such that for all $(v,V) \in \mathbb{H}$ with $\sum_{l=1}^L V_l = 0$, $\|(v,V)\|_* \geq C|V|_{\ell^2}$. By the equivalence of $\|\cdot\|_*$ and $\|\cdot\|_{\dot{\mathbb{H}}}$, Lemma \ref{lem:normequiv}, we have
\[
\|(v,V)\|_* \geq C\inf_{c \in \mathbb{R}}\left(\|v-c\|_{H^1} + |V-c|_{\ell^2}\right) \geq C\inf_{c\in\mathbb{R}}|V-c|_{\ell^2}
\]
The infimum on the right-hand side is attained at
\[
c = \frac{1}{L}\sum_{l=1}^L V_l = 0.
\]
Then by Proposition \ref{prop:fwdcont}, we have
\[
0 \leq |\Pi\circ\mathcal{M}(\sigma_\eps) - \Pi\circ\mathcal{M}(\sigma)|_{\ell^2} \leq \|\mathcal{M}(\sigma_\eps) - \mathcal{M}(\sigma)\|_{*}\rightarrow 0
\]
\end{proof}

\section{The Inverse Problem}
\label{sec:post}
We are interested in the inverse problem of determining the conductivity field from measurements of the voltages $(V_l)_{l=1}^L$ on the boundary, for a variety of input currents $(I_l)_{l=1}^L$ on the boundary. To this end we introduce the following version of Ohm's law. Observe that the mapping $I\mapsto v'$, taking the current stimulation pattern to the solution of (\ref{wpde}), is linear. Then given a conductivity field $\sigma \in \mathcal{A}(D)$, there exists a resistivity matrix  $R(\sigma) \in \mathbb{R}^{L\times L}$ such that the boundary voltage measurements $V(\sigma)$ arising from the solution of the forward model are related to $I$ via
\[
V(\sigma) = R(\sigma)I
\]
By applying several different current stimulation patterns we should be able to infer more about the conductivity $\sigma$. Note however that since the mapping $I\mapsto V$ is linear, only linearly independent stimulation patterns will provide more information\footnote{If there is noise on the measurements, additional linearly dependent observations can be made to effectively reduce the noise level on the original measurements. We can assume that this has been done and scale the noise appropriately.}. Since we have the conservation of charge condition on $I$, there are at most $L-1$ linearly independent patterns we can use. 

Assume that $J$ linearly independent current patterns $I^{(j)} \in \mathbb{R}^L$, $j=1,\ldots,J$, $J \leq L-1$ are applied, and noisy measurements of $V^{(j)} = R(\sigma)I^{(j)}$ are made:
\[
y_j = V^{(j)} + \eta_j,\;\;\;\eta_j \sim N(0,\Gamma_0)\text{ iid.}
\]
We have
\[
y_j = \mathcal{G}_j(\sigma) + \eta_j
\]
where $\mathcal{G}_j(\sigma) = R(\sigma)I^{(j)}$.
Concatenating these observations, we write
\begin{align*}
y = \mathcal{G}(\sigma) + \eta,\;\;\;&\eta \sim N(0,\Gamma)
\end{align*}
where $\Gamma = \mathrm{diag}(\Gamma_0,\ldots,\Gamma_0)$. The inverse problem is then to recover the conductivity field $\sigma$ from the data $y$. This problem is highly ill-posed: the data is finite dimensional, yet we wish to recover a function which, typically, lies in an infinite dimensional space. We take a Bayesian approach by placing a prior distribution on $\sigma$. The choice of prior may have significant effect on the resulting posterior distribution, and different choices of prior may be more appropriate depending upon the prior knowledge of the particular experimental set-up under consideration. 

In subsection \ref{ssec:prior} we outline three different families of prior models, and show the appropriate regularity of the forward maps arising from them. In subsection \ref{ssec:post} we describe the likelihood and posterior distribution formally, before rigorously proving that the posterior distribution exists and is Lipschitz with respect to the data in the Hellinger metric.

\subsection{Choices of Prior}
\label{ssec:prior}
In this section we consider three priors, labelled by $i=1,2,3$, defined by functions $F_i:X_i\rightarrow \mathcal{A}(D)$ which map draws from prior measures on the Banach spaces $X_i$ to the space of conductivities $\mathcal{A}(D)$. Our prior conductivity distributions will then be the pushfoward of the prior measures by these maps $F_i$. We describe these maps, and establish continuity properties of them needed for the study of the posterior later. \mmd{In what follows, the space $C_0(\overline{D})$ of continuous functions on $\overline{D}$ will be equipped with the supremum norm $\|\cdot\|_\infty$ and the corresponding Borel $\sigma$-algebra.}

\subsubsection{Log-Gaussian prior}
We first consider the simple case that the coefficient is given by the exponential of a continuous function. Let $F_1:C^0(\overline{D})\rightarrow \mathcal{A}(D)$ be defined by $F_1(u) = \exp(u)$. Then it is easily seen that $F_1$ does indeed map into $\mathcal{A}(D)$. Furthermore, since $D$ is bounded, if $u \in C^0(\overline{D})$ and $(u_\eps)_{\eps>0}\subseteq C^0(\overline{D})$ is a sequence such that $\|u_\eps - u\|_\infty\rightarrow 0$, then $\|F_1(u_\eps) - F_1(u)\|_\infty \rightarrow 0$.

In this case, we will take our prior measure $\mu_0$ on $u$ to be a \mmd{non-degenerate} Gaussian measure $N(m_0,\mathcal{C}_0)$ on $C^0(\overline{D})$. Note that the push forward of a Gaussian measure by $F_1$ is a log-Gaussian measure. 

\begin{example}
Consider the case $D = B(0,1) \subseteq \mathbb{R}^2$. Suppose that $u$ is drawn from a Gaussian measure $\mu_0 = N(0,\mathcal{C})$. Typical samples from $F_1^\#(\mu_0)$ are shown below\footnote{Given a measure $\mu$ on $(X,\mathcal{X})$ and a measurable map $F:(X,\mathcal{X})\rightarrow (Y,\mathcal{Y})$ between measurable spaces, $F^{\#}(\mu)$ denotes the pushforward of $\mu$ by $F$, i.e. the measure on $(Y,\mathcal{Y})$ given by $F^{\#}(\mu)(A) = \mu(F^{-1}(A))$ for all $A \in \mathcal{Y}$.}. The covariance $\mathcal{C}$ is chosen such that the samples $u$ almost surely have regularity \mmd{$u \in C^{\lfloor s \rfloor, s-\lfloor s \rfloor}(D)$} for all $s < t$, where from left to right $t = 2, 1.5, 1, 0.5$ respectively. Here the samples are generated on $[-1,1]^2\supseteq D$ and then restricted to $D$, for computational simplicity.

\begin{center}
\includegraphics[width=1\textwidth, trim= 4cm 0cm 4cm 0cm]{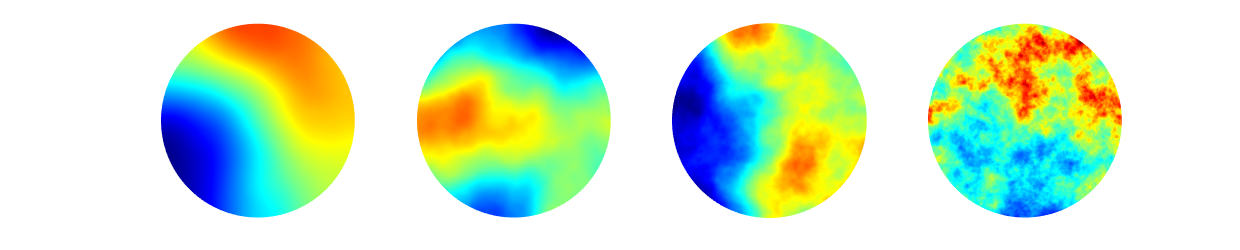}
\end{center}

\end{example}

\subsubsection{Star-shaped prior}
We now consider star-shaped inclusions, that is, inclusions parametrised by their centre and a radial function. These were studied in two-dimensions in the paper \cite{star} to parametrise domains for a Bayesian inverse shape scattering problem. In \cite{star} the authors prove well-posedness of the inverse problem in an infinite dimensional setting through the use of shape derivatives and Riesz-Fredholm theory.

Let $D \subseteq \mathbb{R}^d$, and $R_{d-1} = (-\pi,\pi]\times [0,\pi]^{d-2} \subseteq\mathbb{R}^{d-1}$. Let $h:\mathbb{R}^d\rightarrow R_{d-1}$ be the continuous function representing the mapping from Cartesian to angular polar coordinates. Define the mapping $A:C^0_P(R_{d-1})\times D \rightarrow\mathcal{B}(D)$ by
\[
A(r,x_0) = \left\{x \in D\;\big|\;|x-x_0| \leq r(h(x-x_0))\right\}
\]
where $C^0_P(R_{d-1})$ is the space of continuous periodic functions on $R_{d-1}$. Then $A(r,x_0)$ describes the set of points in $D$ which lie within the closed surface\linebreak parametrised in polar coordinates centred at $x_0$ by
\[
\Gamma(\Theta) = (\Theta,r(\Theta)),\;\;\;\Theta \in R_{d-1}.
\]
In two dimensions, we have $R_1 = (-\pi,\pi]$ and the mapping $h:\mathbb{R}^2\rightarrow R_1$ is given by
\[
h(x,y) = \mathrm{atan2}(y,x) \equiv 2\arctan\left(\frac{y}{\sqrt{x^2+y^2} + x}\right)
\]
where atan2 is the two-parameter inverse tangent function. 

In three dimensions, we have $R_2 = (-\pi,\pi]\times[0,\pi]$ and the mapping $h:\mathbb{R}^3\rightarrow R_2$ is given by
\[
h(x,y,z) = \left(\mathrm{atan2}(y,x),\mathrm{arccot}\left(\frac{z}{\sqrt{x^2+y^2}}\right)\right).
\]

Similar expressions for $h$ exist in higher dimensions, though for applications we are only interested in the case $d = 2, 3$.

Define now the map $F_2:C^0_P(R_{d-1})\times D\rightarrow \mathcal{A}(D)$ by
\begin{align*}
F_2(r,x_0) &= u_+\mathds{1}_{A(r,x_0)} + u_-\mathds{1}_{D\setminus A(r,x_0)}\\
&= (u_+ - u_-)\mathds{1}_{A(r,x_0)} + u_-,
\end{align*}
\mmd{where $u_+, u_- > 0$ are the scalar conductivity values}. Again it can easily be seen that $F_2$ does indeed map into $\mathcal{A}(D)$. We claim that this map is continuous in the following sense:

\begin{proposition}
\label{prop:starcts}
Define the map $F_2:C^0_P(R_{d-1})\times D\rightarrow \mathcal{A}(D)$ as above. Let $x_0 \in D$ and let $r \in C^0_P(R_{d-1})$ be Lipschitz continuous.
\begin{enumerate}[(i)]
\item Suppose that $(r_\eps)_{\eps>0} \subseteq C^0_P(R_{d-1})$ is a sequence of functions such that\linebreak $\|r_\eps - r\|_\infty\rightarrow 0$. Then $F_2(r_\eps,x_0) \rightarrow F_2(r,x_0)$ in measure\footnote{A sequence of functions $(f_\eps)_{\eps>0}$, $f_\eps:D\rightarrow\mathbb{R}$, is said to converge in measure to a function $f:D\rightarrow\mathbb{R}$ if for all $\delta > 0$, $|\{x \in D\;|\;|f_\eps(x) - f(x)| > \delta\}|\rightarrow 0$. Here $|B|$ denotes the Lebesgue measure of a set $B\subseteq\mathbb{R}^d$.}.
\item Suppose that $(x_0^\eps)_{\eps > 0} \subseteq D$ is a sequence of points such that $|x_0^\eps - x_0| \rightarrow 0$. Then $F_2(r,x_0^\eps) \rightarrow F_2(r,x_0)$ in measure.
\item Let $(r_\eps)_{\eps>0}$, $(x_0^\eps)_{\eps>0}$ be as above. Then $F_2(r_\eps,x_0^\eps) \rightarrow F_2(r,x_0)$ in measure.
\end{enumerate}
\end{proposition}

\begin{proof}
In order to show that a sequence of functions $f_\eps:D\rightarrow\mathbb{R}$ converges to $f:D\rightarrow\mathbb{R}$ in measure, it suffices to show that there exists a sequence of sets $Z_\eps\subseteq D$ with $|Z_\eps|\rightarrow 0$ such that $|f_\eps - f| \leq C\mathds{1}_{Z_\eps}$. Then for each $\delta > 0$ we have
\[
|\{x \in D\,|\,|f_\eps(x)-f(x)| > \delta\}| \leq |\{x \in D\,|\,|f_\eps(x) - f(x)| \neq 0\}| \leq |Z_\eps| \rightarrow 0.
\]
\begin{enumerate}[(i)]
\item Fix the centre $x_0 \in D$. Denote $A(r) = A(r,x_0)$. Let $r \in C^0_P(R_{d-1})$ and let $(r_{\eps})_{\eps>0} \subseteq C^0_P(R_{d-1})$ be a sequence of functions such that $\|r_\eps - r\|_{\infty} \rightarrow 0$. Then there exists $\gamma(\eps)\rightarrow 0$ such that $\|r_\eps - r\|_\infty < \gamma(\eps)$. By definition we then have
\[
r(x) - \gamma(\eps) \leq r_\eps(x) \leq r(x) + \gamma(\eps)\;\;\;\text{for all }x \in D\text{ and }\eps > 0.
\]
It follows that we have the inclusions
\begin{align*}
A(r - \gamma(\eps)) \subseteq A(r_\eps) \subseteq A(r + \gamma(\eps)),\\
A(r - \gamma(\eps)) \subseteq A(r) \subseteq A(r + \gamma(\eps)).
\end{align*}
Let $\Delta$ denote the symmetric difference. We deduce that
\[
A(r_\eps)\Delta A(r) \subseteq A(r+\gamma(\eps))\setminus A(r-\gamma(\eps)).
\]
Now the right-hand side is given by
\begin{align*}
A(r+\gamma(\eps))&\setminus A(r-\gamma(\eps))\\
&=\left\{x\in D\;\big|\;r(h(x-x_0)) - \gamma(\eps) < |x-x_0| \leq r(h(x-x_0)) + \gamma(\eps)\right\}.
\end{align*}
As $\eps\rightarrow 0$, this set decreases to the boundary set
\[
\partial A(r) = \left\{x \in D\;\big|\;|x-x_0| = r(h(x-x_0))\right\}.
\]
Since the graph of a continuous function has Lebesgue measure zero, we deduce that $|\partial A(r)| = 0$. It follows that
\[
\lim_{\eps\rightarrow 0} |A(r_\eps)\Delta A(r)| = 0.
\]
To conclude, note that
\[
|F_2(r_\eps,x_0) - F_2(r,x_0)| \leq |u_+ - u_-||\mathds{1}_{A(r_\eps)} - \mathds{1}_{A(r_\eps)}| = C\mathds{1}_{A(r_\eps)\Delta A(r)}.
\]

\item Let $r \in C^0_P(R_{d-1})$ be Lipschitz continuous. Denote $A(x_0) = A(r,x_0)$. Let $(x_0^\eps)\subseteq D$ be a sequence of points such that $|x_0^\eps-x_0|\rightarrow 0$. Note that we may write
\begin{align*}
A(x_0^\eps) &= \{x \in D\;|\;|x-x_0^\eps| \leq r(h(x-x_0^\eps))\}\\
&= \{x \in \mathbb{R}^d\;|\;|x-x_0^\eps| \leq r(h(x-x_0^\eps))\}\cap D\\
&= ((x_0^\eps - x_0) + \{x \in \mathbb{R}^d\;|\;|x-x_0| \leq r(h(x-x_0))\})\cap D\\
&=: ((x_0^\eps - x_0) + A(x_0)^*)\cap D.
\end{align*}
By the distributivity of intersection over symmetric difference, we then have that
\begin{align*}
A(x_0^\eps)\Delta A(x_0) &= [((x_0^\eps - x_0) + A(x_0)^*)\cap D]\Delta [A(x_0)^*\cap D]\\
&= [((x_0^\eps - x_0) + A(x_0)^*)\Delta A(x_0)^*]\cap D\\
& \subseteq ((x_0^\eps - x_0) + A(x_0)^*)\Delta A(x_0)^*.
\end{align*}
Therefore, using Theorem 1 from \cite{geometry}, we see that
\begin{align*}
|A(x_0^\eps)\Delta A(x_0)| &\leq |((x_0^\eps - x_0) + A(x_0)^*)\Delta A(x_0)^*|\\
& \leq |x_0^\eps - x_0|\mathcal{H}^{d-1}(\partial A(x_0)^*)
\end{align*}
where $\mathcal{H}^{d-1}$ is the $(d-1)$-dimensional Hausdorff measure. Since we assume that $r$ is Lipschitz, the surface area $\mathcal{H}^{d-1}(\partial A(x_0)^*)$ of the boundary of $A(x_0)^*$ is finite, and so it follows that
\[
\lim_{\eps\rightarrow 0} |A(x_0^\eps)\Delta A(x_0)| = 0.
\]
As before, we conclude by noting that
\[
|F_2(r,x_0^\eps) - F_2(r,x_0)| \leq |u_+ - u_-||\mathds{1}_{A(x_0^\eps)} - \mathds{1}_{A(x_0)}| = C\mathds{1}_{A(x_0^\eps)\Delta A(x_0)}.
\]

\item We have that
\begin{align*}
|F_2(r_\eps,x_0^\eps) - F_2(r,x_0)| &\leq |F_2(r_\eps,x_0^\eps) - F_2(r,x_0^\eps)| + |F_2(r,x_0^\eps) - F_2(r,x_0)|\\
&\leq C(\mathds{1}_{A(r_\eps,x_0^\eps)\Delta A(r,x_0^\eps)} + \mathds{1}_{A(r,x_0^\eps)\Delta A(r,x_0)})\\
&\leq C\mathds{1}_{[A(r_\eps,x_0^\eps)\Delta A(r,x_0^\eps)]\cup [A(r,x_0^\eps)\Delta A(r,x_0)]}.
\end{align*}
Now note that
\[
|A(r_\eps,y_0)\Delta A(r,y_0)| \leq |A(r_\eps,y_0)^*\Delta A(r,y_0)^*|.
\]
The right hand-side is independent of $y_0$ by translation invariance of the Lebesgue measure. By the same argument as part (i) we conclude that it tends to zero. We then have that
\begin{align*}
|[A(r_\eps,x_0^\eps)\Delta A(r,x_0^\eps)]&\cup [A(r,x_0^\eps)\Delta A(r,x_0)]|\\
&\leq |A(r_\eps,x_0^\eps)\Delta A(r,x_0^\eps)| + |A(r,x_0^\eps)\Delta A(r,x_0)|\\
&\leq \sup_{y_0 \in D} |A(r_\eps,y_0)\Delta A(r,y_0)| + |A(r,x_0^\eps)\Delta A(r,x_0)|
\end{align*}
which tends to zero by the discussion above and part (ii).
\end{enumerate}
\end{proof}

\begin{remark}
Above we assumed that $r:R_{d-1}\rightarrow\mathbb{R}$ was Lipschitz continuous. This assumption is only used in the proof of part (ii) of the proposition. If the centre of the star-shaped region is known, this assumption may then be dropped to allow for rougher boundaries.
\end{remark}

We need to choose a prior measure $\mu_0$ on $(r,x_0)$. \mmd{We equip $C^0_P(R_{d-1})\times D$ with the norm $\|(r,x_0)\| = \max\{\|r\|_\infty,|x_0|\}$ and corresponding Borel $\sigma$-algebra.} We assume that $r$ and $x_0$ are independent under the prior so that we may factor $\mu_0 = \sigma_0\otimes\tau_0$ where $\sigma_0$ is a measure on $C^0_P(R_{d-1})$ and $\tau_0$ is a measure on $D$. We will assume that $\sigma_0$ is such that $\sigma_0(B) > 0$ for all balls $B\subseteq C_P^0(R_{d-1})$. 

\begin{example}
Consider the case $D = B(0,1) \subseteq \mathbb{R}^2$. Suppose that $r$ is drawn from a log-Gaussian measure $\sigma_0$ on $C^0_P((-\pi,\pi])$, and $x_0$ is drawn from $\tau_0 = U([-0.5,0.5]^2)$. Note that $[-0.5,0.5]^2 \subseteq B(0,1)$. Typical samples from $F_2^\#(\mu_0)$ are shown below. The covariance of $\sigma_0$ is chosen such that the samples $r$ almost surely have regularity \mmd{$r \in C^{\lfloor s \rfloor, s-\lfloor s \rfloor}((-\pi,\pi])$} for all $s < t$, where from left to right $t = 2.5, 2, 1.5, 1$ respectively.
\begin{center}
\includegraphics[width=1\textwidth, trim= 4cm 0cm 4cm 0cm]{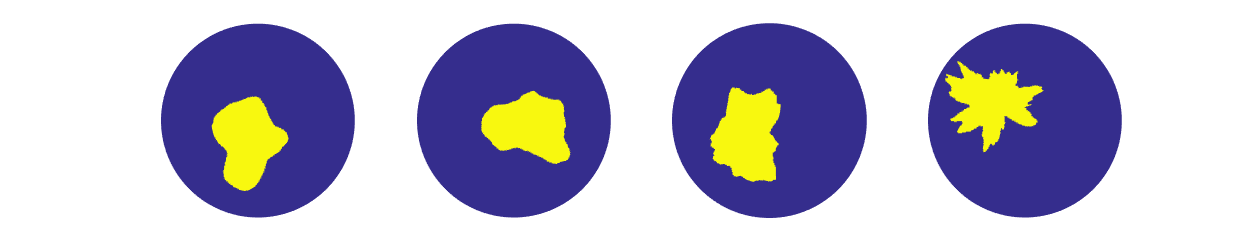}
\end{center}

\end{example}

\subsubsection{Level set prior}
We finally consider the case where the inclusions can be described by a single level set function, as in \cite{levelset}. Let $n \in \mathbb{N}$ and fix constants $-\infty = c_0 < c_1 < \ldots < c_n = \infty$. Given $u:D\rightarrow\mathbb{R}$, define $D_i\subseteq D$ by
\[
D_i = \{x \in D\,|\, c_{i-1} \leq u(x) < c_i\},\;\;\;i=1,\ldots,n
\]
so that $\overline{D} = \bigcup_{i=1}^n \overline{D}_i$ and $D_i\cap D_j = \varnothing$ for $i\neq j$, $i,j \geq 1$. Define also the level sets
\[
D_i^0 = \overline{D}_i\cap \overline{D}_{i+1} = \{x \in D\,|\,u(x) = c_i\},\;\;\; i=1,\ldots,n-1.
\]
Now given strictly positive functions $f_1,\ldots,f_n \in C^0(\overline{D})$, we define the map $F_3:C^0(\overline{D})\rightarrow\mathcal{A}(D)$ by
\[
F_3(u) = \sum_{i=1}^n f_i\mathds{1}_{D_i}.
\]
Since each $f$ is continuous and strictly positive on a compact set $\overline{D}$, they are uniformly bounded above and below by positive constants, and so $F_3$ does indeed map into $\mathcal{A}(D)$. 

In this paper we are primarily concerned with the case of binary fields, $n = 2$ and $f_i$ constant above, however the theory in proved in the general case. We have the following result regarding continuity of this map, by the same arguments as in \cite{levelset}.
\begin{proposition}
\label{prop:lvlcts}
Define the map $F_3:C^0(\overline{D})\rightarrow\mathcal{A}(D)$ as above. Let $u \in C^0(\overline{D})$ be such that $|D_i^0| = 0$ for $i=1,\ldots,n-1$. Suppose that $(u_\eps)_{\eps > 0} \subseteq C^0(\overline{D})$ is an approximating sequence of functions so that $\|u_\eps - u\|_\infty \rightarrow 0$. Then $F_3(u_\eps) \rightarrow F_3(u)$ in measure.
\end{proposition}

\begin{proof}
Denote by $D_{i,\eps}$ and $D_{i,\eps}^0$ the sets as defined above associated with the approximating functions $u_\eps$. We can write
\[
F_3(u_\eps) - F_3(u) = \sum_{i=1}^n\sum_{j=1}^n (f_i - f_j)\mathds{1}_{D_{i,\eps}\cap D_j} = \sum_{\substack{i,j=1\\i\neq j}}^n (f_i - f_j)\mathds{1}_{D_{i,\eps}\cap D_j}.
\]
Since $\|u_\eps - u\|_\infty\rightarrow 0$, there exists $\gamma(\eps)\rightarrow 0$ with $\|u_\eps - u\|_\infty < \gamma(\eps)$. Then we have for all $x \in D$ and $\eps > 0$
\[
u(x) - \gamma(\eps) < u_\eps(x) < u(x) + \gamma(\eps).
\]
Hence for $|j-i| > 1$ and $\eps$ sufficiently small, $D_{i,\eps}\cap D_j = \varnothing$. If $|j-i| = 1$, then
\begin{align*}
D_{i,\eps}\cap D_{i+1} &\subseteq \widetilde{D}_{i,\eps} := \{x \in D\,|\,c_i \leq u(x) < c_i + \gamma(\eps)\}\rightarrow D_i^0,\\
D_{i,\eps}\cap D_{i-1} &\subseteq \hat{D}_{i-1,\eps} := \{x \in D\,|\,c_i-\gamma(\eps) \leq u(x) < c_i\}\rightarrow\varnothing.
\end{align*}
By the uniform boundedness of the $(f_i)$, for sufficiently small $\eps$ we can then write
\begin{align}
|F_3(u_\eps) - F_3(u)| &\leq \sum_{i=1}^{n-1} |f_i - f_{i+1}|\mathds{1}_{\widetilde{D}_{i,\eps}} + \sum_{i=2}^n |f_i - f_{i-1}|\mathds{1}_{\hat{D}_{i-1,\eps}}\notag\\
&\leq C\mathds{1}_{Z_\eps}\label{eq:flevbound}
\end{align}
where $Z_{\eps}\subseteq D$ is given by
\begin{align*}
Z_\eps &= \left(\bigcup_{i=1}^{n-1} \widetilde{D}_{i,\eps}\right)\cup\left(\bigcup_{i=2}^n \hat{D}_{i-1,\eps}\right) \rightarrow \bigcup_{i=1}^{n-1} D_i^0.
\end{align*}
By the assumption that $|D_i^0| = 0$ for all $i$, it follows that $|Z_\eps|\rightarrow 0$, and so the result follows from the comment at the start of the proof of Proposition \ref{prop:starcts}.
\end{proof}

\begin{figure}
\begin{center}
\includegraphics[width=0.92\textwidth, trim= 0cm 0cm 0cm 0cm]{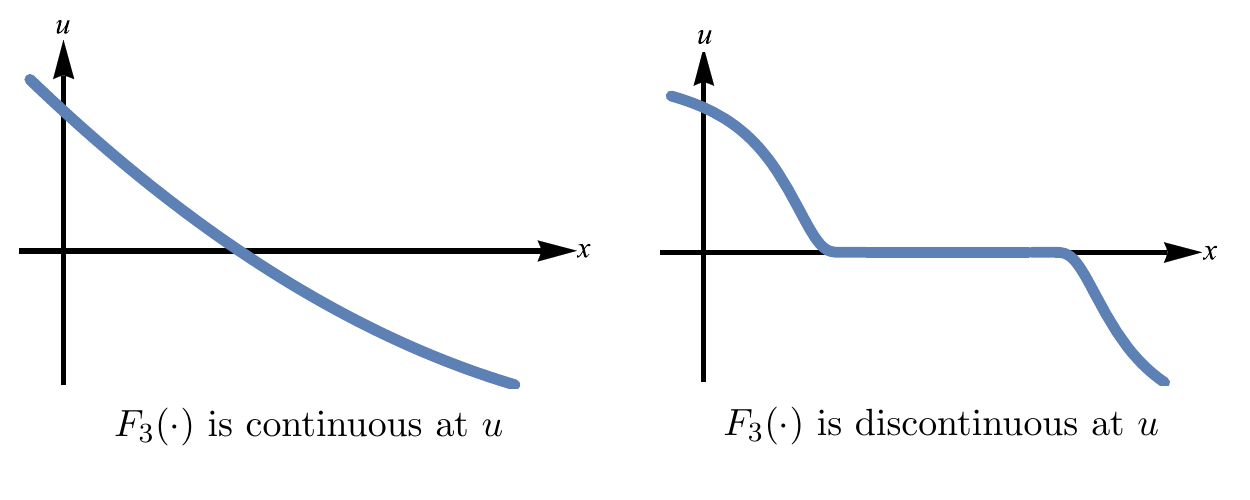}
\caption{The discontinuity of $F_3$ into $L^p(D)$.}.
\label{fig:levelsets}
\end{center}
\end{figure}

Note that bound (\ref{eq:flevbound}) actually above implies the slightly stronger result that, when the $c_i$-level sets of $u \in X$ have zero measure, $F_3$ is continuous into $L^p(D)$, $1\leq p < \infty$, at $u$. The assumption that the level sets have zero measure is an important one, as illustrated by Figure \ref{fig:levelsets}: an arbitrarily small perturbation of $u$ can lead to an order 1 change in $F_3(u)$.

In the Bayesian approach we are taking to this problem, we may choose a prior measure on $u$ such that, almost surely, the Lebesgue measure of the level sets is zero. This is shown to hold for \mmd{non-degenerate} Gaussian measures in \cite{levelset}. As a result, $F_3$ will be almost surely continuous under the prior, and this is enough to give the measurability required in Bayes' theorem, as shown in \cite{levelset}.

As in the log-Gaussian case, we take our prior measure $\mu_0$ on $u$ to be a Gaussian measure $N(m_0,\mathcal{C}_0)$ on $C^0(\overline{D})$.

\begin{example}
Consider the case $D = B(0,1) \subseteq \mathbb{R}^2$, $n = 2$, $c_1 = 0$, $f_1 \equiv 1$ and $f_2 \equiv 2$. Suppose that $u$ is drawn from a centred Gaussian measure $\mu_0 = N(0,\mathcal{C})$ on $C^0(\overline{D})$. The covariance $\mathcal{C}$ is chosen such that the samples $u$ almost surely have regularity \mmd{$u \in C^{\lfloor s \rfloor, s-\lfloor s \rfloor}(D)$} for all $s < t$, where from left to right $t = 4, 3, 2, 1$ respectively. As in the log-Gaussian case, here the samples are generated on $[-1,1]^2\supseteq D$ and then restricted to $D$, for computational simplicity.

\begin{center}
\includegraphics[width=1\textwidth, trim= 4cm 0cm 4cm 0cm]{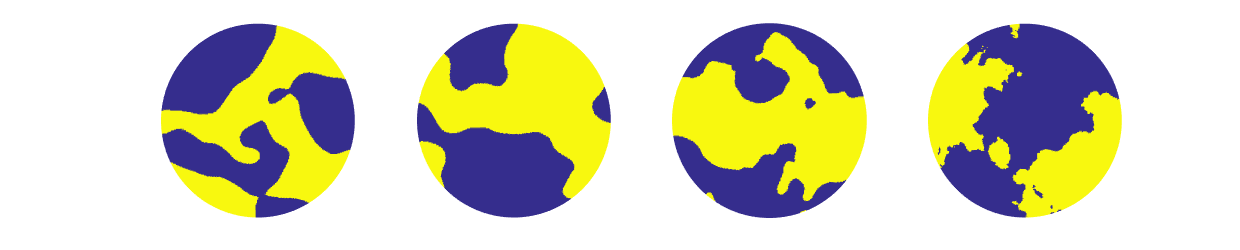}
\end{center}

\end{example}

\mmd{\begin{remark}
In the cases of the star-shaped and level set priors, we have assumed the values that the conductivity takes are known a priori. This may be appropriate in certain situations, for example in medical contexts where conductivities of different types of tissue are known \cite{borcea_review}. However it may be preferable to have the flexibility of treating the conductivity values as part of the inference. The theory does not become significantly more involved, as it can be shown that the conclusions of propositions \ref{prop:starcts} and \ref{prop:lvlcts} still hold when we allow for perturbation of the conductivity values as well. We work with the case of fixed conductivity values for clarity of presentation, but we note that it is possible to place a prior upon these.
\end{remark}
}

\subsection{The Likelihood and Posterior Distribution}
\label{ssec:post}
The inverse problem was introduced at the beginning of the section. Now that we have introduced prior distributions, we may provide the Bayesian formulation of the problem.

Let $X$ be a separable Banach space and $F:X\rightarrow\mathcal{A}(D)$ a map from the space $X$ where the unknown parameters live to the conductivity space. Choose a set of current stimulation patterns $I^{(j)} \in \mathbb{R}^L$, $j=1,\ldots,J$ and let $\mathcal{M}_j:\mathcal{A}(D)\rightarrow\mathbb{H}$ denote the solution map when using stimulation pattern $I^{(j)}$. Recall the projection map $\Pi:\mathbb{H}\rightarrow\mathbb{R}^L$ was defined by $\Pi(v,V) = V$.

The data $y_j$ from the $j$th stimulation pattern is assumed to arise from the map $\mathcal{G}_j:X\rightarrow\mathbb{R}^L$, $\mathcal{G}_j = \Pi\circ\mathcal{M}_j\circ F$, via
\[
y_j = \mathcal{G}_j(u) + \eta_j,\;\;\;\eta_j \sim N(0,\Gamma_0)\text{ iid.}
\]
We concatenate these observations to get data $y \in \mathbb{R}^{JL}$ given by
\begin{align*}
y = \mathcal{G}(u) + \eta,\;\;\;&\eta \sim \mathbb{Q}_0 := N(0,\Gamma)
\end{align*}
where $\Gamma = \mathrm{diag}(\Gamma_0,\ldots,\Gamma_0)$ and $\mathcal{G}:X\rightarrow\mathbb{R}^{JL}$. This coincides with the setup at the start of the section, with $\sigma = F(u)$.

Assume that $u \sim \mu_0$, where $\mu_0$ is independent of $\mathbb{Q}_0$. From the above, we see that $y|u \sim \mathbb{Q}_u:= N(\mathcal{G}(u),\Gamma)$. We use this to find the distribution of $u|y$. First note that
\begin{align*}
\frac{\dee\mathbb{Q}_u}{\dee \mathbb{Q}_0}(y) = \exp\left(-\Phi(u;y) + \frac{1}{2}|y|_\Gamma^2\right)
\end{align*}
where the potential (or negative log-likelihood) $\Phi:X\times Y\rightarrow\mathbb{R}$ is given by
\begin{align}
\label{eq:phi}
\Phi(u;y) = \frac{1}{2}|\mathcal{G}(u) - y|_\Gamma^2 := \mmd{\frac{1}{2}|\Gamma^{-\frac{1}{2}}(\mathcal{G}(u)-y)|^2.}
\end{align}
Then under suitable regularity conditions, Bayes' theorem tells us that the distribution $\mu^y$ of $u|y$ is as given below:
\begin{theorem}[Existence and Well-Posedness]
\label{thm:existence}
Let $(X,\mathcal{F},\mu_0)$ denote any of the probability spaces associated with any of the three priors introduced in the previous subsection, and let $\Phi:X\times Y\rightarrow\mathbb{R}$ be the potential (\ref{eq:phi}) associated with the corresponding forward map. Then the posterior distribution $\mu^y$ of the state $u$ given data $y$ is well-defined. Furthermore, $\mu^y\ll \mu_0$ with Radon-Nikodym derivative
\begin{align}
\label{eq:posterior}
\frac{\dee \mu^y}{\dee \mu_0}(u) = \frac{1}{Z_\mu}\exp(-\Phi(u;y))
\end{align}
where for $y$ $\mathbb{Q}_0$-a.s.,
\[
Z_\mu := \int_X \exp(-\Phi(u;y))\,\mu_0(\dee u) > 0.
\]
Additionally, the posterior measure $\mu^y$ is locally Lipschitz with respect to $y$, in the Hellinger distance: for all $y,y' \in Y$ with $\max\{|y|_\Gamma,|y'|_\Gamma\} < \rho$, there exists $C = C(\rho) > 0$ such that
\[
d_{\mathrm{Hell}}(\mu^y,\mu^{y'}) \leq C|y-y'|_\Gamma.
\]
\end{theorem}

In the proof of the above theorem we will make use of the following version of Bayes' theorem from \cite{lecturenotes}.

\begin{proposition}[Bayes' theorem]
\label{thm:bayes}
Define the measure $\nu_0(\dee u, \dee y) = \mu_0(\dee u)\mathbb{Q}_0(\dee y)$ on $X\times Y$. Assume that $\Phi:X\times Y\rightarrow\mathbb{R}$ is $\nu_0$-measurable and that, for $y$ $\mathbb{Q}$-a.s.
\[
Z_\mu = \int_X \exp(-\Phi(u;y))\,\mu_0(\dee u) > 0.
\]
Then the conditional distribution of $u|y$ exists and is denoted by $\mu^y$. Furthermore $\mu^y\ll\mu_0$ and, for $y$ $\mathbb{Q}_0$-a.s.,
\[
\frac{\dee \mu^y}{\dee \mu_0} = \frac{1}{Z_\mu}\exp\left(-\Phi(u;y)\right).
\]
\end{proposition}

We need to verify that the assumptions of this theorem are satisfied. To proceed we first give some regularity properties of the potential $\Phi$:
\begin{proposition}
\label{prop:phireg}
Let $(X,\mathcal{F},\mu_0)$ denote any of the probability spaces associated with the priors introduced in the previous subsection. Then the potential $\Phi:X\times Y\rightarrow \mathbb{R}$ associated with the corresponding forward map, given by (\ref{eq:phi}), admits the following properties.
\begin{enumerate}[(i)]
\item There is a continuous $K:\mathbb{R}^+\times\mathbb{R}^+\rightarrow\mathbb{R}^+$ such that for every $\rho > 0$, $u \in X$ and $y \in Y$ with $|y|_\Gamma < \rho$,
\[
0 \leq \Phi(u;y) \leq K(\rho,\|u\|_X).
\]
In the cases $F = F_{2}$ and $F = F_{3}$, $K$ has no dependence on $\|u\|_X$.
\item For any fixed $y \in Y$, $\Phi(\cdot;y):X\rightarrow\mathbb{R}$ is continuous $\mu_0$-almost surely on the probability space $(X,\mathcal{F},\mu_0)$.
\item There exists $C:\mathbb{R}^+\times\mathbb{R}^+\rightarrow\mathbb{R}^+$ such that for every $y_1, y_2 \in Y$ with\linebreak $\max\{|y_1|_\Gamma,|y_2|_\Gamma\} < \rho$, and every $u \in X$, 
\[
|\Phi(u;y_1) - \Phi(u;y_2)| \leq C(\rho,\|u\|_X)|y_1 - y_2|_\Gamma.
\]
Moreover, $C(\rho,\|\cdot\|_X) \in L^2_{\mu_0}(X)$ for all $\rho > 0$.
\end{enumerate}
\end{proposition}

\begin{proof}
\begin{enumerate}[(i)]
\item From equation (\ref{eq:propcts2}) in the proof of Proposition \ref{prop:fwdcont}, we see that there exists $C > 0$ such that each $\mathcal{M}_j:\mathcal{A}(D)\rightarrow\mathbb{H}$ satisfies
\[
\|\mathcal{M}_j(\sigma_1) - \mathcal{M}_j(\sigma_2)\|_* \leq C\|\mathcal{M}_j(\sigma_2)\|_*\|\sigma_1 - \sigma_2\|_\infty
\]
for all $\sigma_1,\sigma_2 \in \mathcal{A}(D)$. Taking $\sigma_2 \equiv 1$, say, we deduce that
\[
\|\mathcal{M}_j(\sigma_1)\|_* \leq C\|\mathcal{M}_j(1)\|_*\|\sigma_1 - 1\|_\infty + \|\mathcal{M}_j(1)\|_* \leq C(1+\|\sigma_1\|_\infty).
\]
Hence $\|\sigma\|_\infty < \rho$ implies that $\|\mathcal{M}_j(\sigma)\|_* < C(1+\rho)$. By Corollary \ref{cor:fwdcont}, it follows that $\Pi\circ\mathcal{M}_j:\mathcal{A}(D)\rightarrow\mathbb{R}^L$ is bounded on bounded sets with respect to $\|\cdot\|_\infty$ for all $j$.

In the case $F = F_1$, if $u \in X$ then $\|F(u)\|_\infty \leq e^{\|u\|_X}$. It follows that $|\mathcal{G}(u)|_\Gamma \leq \max_j |\mathcal{G}_j(u)|_\Gamma \leq C(1+e^{\|u\|_X})$.

Now note that
\[
\Phi(u;y) \leq |\mathcal{G}(u)|_\Gamma^2 + |y|_\Gamma^2. 
\]
Then for any $y \in Y$ with $|y| < \rho$, we may bound
\[
\Phi(u;y) \leq C(1 + e^{2\|u\|_X} + \rho^2) =: K(\rho,\|u\|_X).
\]
In the cases $F = F_{2}$ and $F = F_{3}$, we have that $\|F(u)\|_\infty$ is bounded uniformly over $u \in X$ and so $|\mathcal{G}(u)|_\Gamma \leq \max_j|\mathcal{G}_j(u)|_\Gamma \leq C$. Hence we obtain the bound
\[
\Phi(u;y) \leq C(1+\rho^2) =: K(\rho).
\]

\item Let $u \sim \mu_0$ and suppose $F:X\rightarrow\mathcal{A}(D)$ is such that $\|u_\eps-u\|_X \rightarrow 0$ implies that $F(u_\eps)\rightarrow F(u)$ either uniformly or in measure. Then Proposition \ref{prop:fwdcont} tells us that $\mathcal{M}_j\circ F:X\rightarrow\mathbb{H}$ is continuous at $u$ for each $j$. The projection $\Pi:\mathbb{H}\rightarrow\mathbb{R}^L$ is continuous, and so $\mathcal{G}_j = \Pi\circ\mathcal{M}_j\circ F$ is continuous at $u$ for each $j$. In \S\ref{ssec:prior} it is shown that this is true for $F = F_1$ and $F = F_2$ for any $u$. For $F = F_{3}$ it is only true at points $u$ whose level sets have zero measure, however since we are assuming $u \sim \mu_0$, a Gaussian measure, it follows from Proposition 7.2 in \cite{levelset} that $u$ $\mu_0$-almost surely has this property.
\item Let $u \in X$ and $y_1,y_2 \in Y$ with $\max\{|y_1|_\Gamma,|y_2|_\Gamma\} < \rho$. Then we have
\begin{align*}
|\Phi(u;y_1) - \Phi(u;y_2)| &= \frac{1}{2}|\langle y_1+y_2-2\mathcal{G}(u),y_1-y_2\rangle_\Gamma|\\
&\leq \frac{1}{2}(|y_1|_\Gamma+|y_2|_\Gamma+2|\mathcal{G}(u)|_\Gamma)|y_1-y_2|_\Gamma\\
&\leq (\rho + |\mathcal{G}(u)|_\Gamma)|y_1-y_2|_\Gamma\\
&=: \mmd{\tilde{C}(\rho,u)|y_1-y_2|_\Gamma}
\end{align*}
We now consider cases separately based on the prior. In the log-Gaussian case, we may bound
\[
\mmd{\tilde{C}(\rho,u) \leq C(1+ \rho + e^{\|u\|_X}) =: C(\rho,\|u\|_X)}
\]
using the bound from the proof of part (i). Square-integrability of $C(\rho,\|\cdot\|_X)$ follows since Gaussians have exponential moments.

In the star-shaped and level set prior cases, we have that $|\mathcal{G}(u)|$ is bounded uniformly by a constant. We may hence bound $\tilde{C}(\rho,u)$ above by some\linebreak $C(\rho,\|u\|_X) := C(1+\rho)$ that is independent of $u$, and so again the square-integrability follows.
\end{enumerate}
\end{proof}

\begin{proof}[Proof of Theorem \ref{thm:existence}]
Define the product measure $\nu_0(\dee u, \dee y) = \mu_0(\dee u)\mathbb{Q}_0(\dee y)$ on $X\times Y$. We showed in Proposition \ref{prop:phireg} that $\Phi(\cdot;y):X\rightarrow\mathbb{R}$ is almost-surely continuous under the prior for all $y \in Y$, and $\Phi(u;\cdot):Y\rightarrow\mathbb{R}$ is locally Lipschitz for all $u \in X$. Together these imply that $\Phi:X\times Y\rightarrow\mathbb{R}$ is almost-surely jointly continuous under $\nu_0$. To see this, let $(u,y) \in X\times Y$ and let $(u_n,y_n)_{n\geq 1} \subseteq X\times Y$ be an approximating sequence so that $\|u_n-u\|_X \rightarrow 0$ and $|y_n-y|_\Gamma\rightarrow 0$. Then we have
\begin{align*}
|\Phi(u_n,y_n) - \Phi(u,y)| &\leq |\Phi(u_n,y_n) - \Phi(u_n,y)| + |\Phi(u_n,y) - \Phi(u,y)|.
\end{align*}
The second term tends to zero $\mu_0$-almost surely by continuity. For the first term, note that the sequences $(\|u_n\|_X)_{n\geq 1}$ and $(|y_n|_\Gamma)_{n\geq 1}$ are bounded, by $K$ and $R$ respectively, say. Then we can use the local Lipschitz property to deduce that
\begin{align*}
|\Phi(u_n,y_n) - \Phi(u_n,y)| &\leq C(R,K)|y_n-y|_\Gamma
\end{align*}
since $C(\cdot,\cdot):\mathbb{R}\times\mathbb{R}\rightarrow\mathbb{R}$, \mmd{as defined in the proof of Proposition \ref{prop:phireg}}, is monotonically increasing in both components. Therefore this term tends to zero, and we obtain the desired continuity. It follows, see for example Lemma 6.1 in \cite{levelset}, that $\Phi$ is $\nu_0$-measurable.

For the lower bound on $Z_\mu$, we consider cases separately based on the prior. First we consider the log-Gaussian and level set prior cases so that $\mu_0$ is Gaussian. Let $B \subseteq X$ be any ball. Fix any $\rho > |y|_\Gamma$ and define
\[
R = \sup_{u \in B} K(\rho,\|u\|_X)
\]
where $K$ is the upper bound from Proposition \ref{prop:phireg}(i). This supremum is finite by the continuity of $K$. Then we have
\begin{align*}
\int_X \exp(-\Phi(u;y))\,\mu_0(\dee u) &\geq \int_B \exp(-\Phi(u;y))\,\mu_0(\dee u)\\
&\geq \int_B \exp(-K(\rho,\|u\|))\,\mu_0(\dee u)\\
&\geq \exp(-R)\mu_0(B).
\end{align*}
Since $\mu_0$ is Gaussian, $\mu_0(B) > 0$ and so $Z_\mu > 0$.

In the star-shaped prior case, proceed as above but take $B = B_1\times D$ where $B_1 \subseteq C^0_P(R_{d-1})$ is any ball. Then we have
\[
\mu_0(B) = (\sigma_0\times\tau_0)(B_1\times D) = \sigma_0(B_1)\tau_0(D) > 0
\]
by the assumption that $\sigma_0$ assigns positive mass to balls, and so again $Z_\mu > 0$. The above hold for all $y \in Y$, and so in particular for $y$ $\mathbb{Q}_0$-almost-surely. We may now apply Bayes' Theorem \ref{thm:bayes} to obtain the existence of $\mu^y$.

The proof of well-posedness is almost identical to that of the analogous result Theorem 2.2 in \cite{levelset} and is hence omitted.
\end{proof}

\section{Numerical Experiments}
\label{sec:num}
We investigate the effect of the choice of prior on the recovery of certain binary conductivity fields. The specific fields we consider are shown in Figure \ref{fig:truth}, where blue represents a conductivity of $1$ and yellow a conductivity of $2$. Simulations are performed using the EIDORS software \cite{eidors} to solve the forward model \mmd{using a first order finite element method}; a mesh of $43264$ elements is used to create the data and a mesh of $10816$ elements is used for simulations in order to avoid an inverse crime \cite{kaipio}.

In subsection \ref{ssec:mcmc} we describe the MCMC sampling algorithm that we will use. In subsection \ref{ssec:setup} we define the parameters we will use for the forward model and the MCMC simulations. We also describe how the data is created, and define our choices of prior distributions. Finally in subsection \ref{ssec:results} we present the results of the simulation, looking at quality of reconstruction, convergence of the algorithm and some properties of the posterior distribution.

\begin{figure}
\begin{subfigure}{.49\textwidth}
\begin{center}
\includegraphics[width=\textwidth]{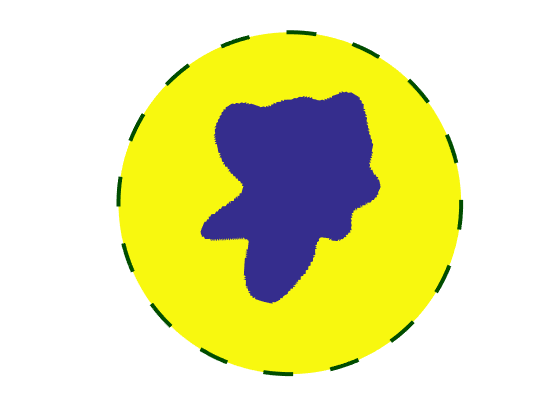}
\caption{Conductivity A}
\end{center}
\end{subfigure}
\begin{subfigure}{.49\textwidth}
\begin{center}
\includegraphics[width=\textwidth]{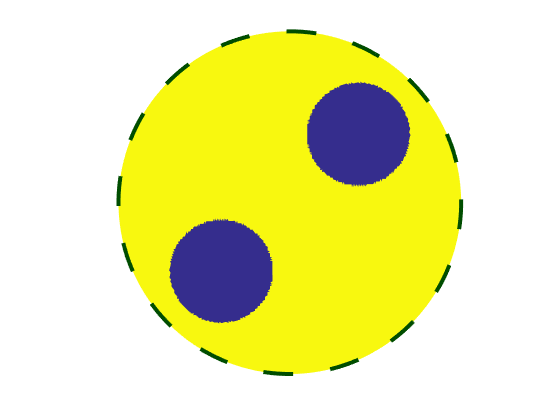}
\caption{Conductivity B}
\end{center}
\end{subfigure}
\caption{The two binary fields we attempt to recover. Conductivity A is drawn from the star-shaped prior, with $\sigma_0 = h^{\#}\big[N(0.5,10^9\cdot(30^2 - \mathcal{A}_D)^{-3})\big]$, $h(z) = (1+\tanh{z})/2$, and $\tau_0 = U([-0.5,0.5]^2)$. Conductivity B is constructed explicitly, rather than being drawn from a prior.}
\label{fig:truth}
\end{figure}

\subsection{Sampling Algorithm}
\label{ssec:mcmc}
We aim to produce a sequence of samples from $\mu^y$ on $X$, where $\mu^y$ is given by (\ref{eq:posterior}). We make use of the preconditioned Crank-Nicolson Markov Chain Monte Carlo (pCN-MCMC) method. The pCN-MCMC method is a modification of the standard Random Walk Metropolis \mmd{(RWM)} MCMC method which is well-adapted to Gaussian priors in high dimensions. It was introduced in \cite{diff_bridges}, and its dimension independent properties are analysed and illustrated numerically in \cite{spectralgap} and \cite{Cotter2013} respectively; the pCN nomenclature was introduced in \cite{Cotter2013}. In the case of the star-shaped prior, we use a Metropolis-within-Gibbs algorithm \cite{mwg}, alternately updating the field with the pCN method above and updating the centre with the standard RWM method.

An advantage of these MCMC methods is that derivatives of the forward map are not needed, only black-box solution of the forward model. However in order to accurately compute some quantity of interest, such as the conditional mean, we may need to produce a very large number of samples and tuning the algorithm to minimise this effect is important. For this reason we compute the effective sample size from the integrated autocorrelation (neglecting a burn-in period) of a quantity of interest, as in \cite{mcmc_kass}.

\subsection{Data and Parameters}
\label{ssec:setup}
We work on a circular domain of radius $1$, with 16 equally spaced electrodes on its boundary providing 50\% coverage. We take all contact impedances $z_l = 0.01$. We stimulate adjacent electrodes with a current of 0.1, so that the matrix of stimulation patterns $I = (I^{(j)})_{j=1}^{15} \in \mathbb{R}^{16\times 15}$ is given by
\[
I = 0.1\times\left(
\begin{array}{cccc}
+1 & 0 & \cdots & 0\\
-1 & +1 &\cdots & 0\\
0 & -1 & \ddots & 0\\
\vdots & \vdots & \ddots & +1\\
 0 & 0 & 0 & -1
\end{array}\right)
\]
The conductivity is chosen such that it takes values $1$ and $2$. We perturb the measurements with white noise $\eta \sim N(0,\gamma^2I)$, $\gamma = 0.0002$, so that the mean relative error on both sets of data is approximately 10\%. The true conductivity fields used to generate the data, henceforth referred to as \emph{Conductivity A} and \emph{Conductivity B}, are shown in Figure \ref{fig:truth}. In all cases we generate $N = 2.5\times 10^6$ samples with a burn-in of $k_0 = 5\times 10^5$ samples.

Our priors on fields will make use of Gaussians with covariances of the form
\begin{align}
\label{eq:wmprior}
\mathcal{C} = q(\tau^2 - \Delta)^{-\alpha}.
\end{align}
These are essentially rescaled Whittle-Matern covariances \cite{whittlematern}, with $\tau$ representing the inverse length scale of the samples, $\alpha$ proportional to their regularity, and $q$ proportional to their amplitude.

In what follows, denote by $\mathcal{A}_N$ the Laplacian with Neumann boundary conditions on $[-1,1]^2$, restricted to $D$, so that its domain is given by
\[
\mathcal{D}(\mathcal{A}_N) = \left\{u|_D\;\bigg|\; u\in H^2([-1,1]^2), \frac{\partial u}{\partial n} = 0\right\}.
\]
Defining the Laplacian first on a square and then restricting to $D$ will allow for efficient generation of Gaussian samples via the fast Fourier transform. Note that if we were to consider priors of the form (\ref{eq:wmprior}) with $\tau = 0$, we should restrict $\mathcal{D}(\mathcal{A}_N)$ further to ensure the invertibility of $\mathcal{A}_N$.
 
Additionally, denote by $\mathcal{A}_D$ the Laplacian with Dirichlet boundary conditions on $R_1 = (-\pi,\pi]$, so that its domain is given by
\[
\mathcal{D}(\mathcal{A}_D) = \big\{u \in H^2\big((-\pi,\pi]\big)\;\big|\; u(-\pi) = u(\pi) = 0\big\}.
\]

\subsubsection{Gaussian prior}
States are defined on a grid of $2^7\times2^7$ points. For both simulations the pCN jump parameter $\beta$ is taken to be 0.01, with choice of prior
\[
\mu_0 = \exp^\#\big[ N(0.5\log{2},10^{16}\cdot(40^2 - \mathcal{A}_N)^{-6})\big].
\]

\subsubsection{Star-shaped prior}
Radial states are defined on a grid of $2^8$ points. For  Conductivity A, we choose the pCN jump parameter $\beta = 0.03$ and the RWM jump parameter $\delta = 0.01$. For Conductivity B we choose $\beta = 0.01$ and $\delta = 0.005$. For both simulations we use the choice of prior $\mu_0 = \sigma_0\times\tau_0$, with
\[
\sigma_0 = h^\#\big[N(0.5,10^{9}\cdot(30^2-\mathcal{A}_D)^{-3})\big],\;\;\;\tau_0 = U([-0.5,0.5]^2),
\]
where $h(z) = (1+\tanh{z})/2$.

\subsubsection{Level set prior}
States are defined on a grid of $2^7\times2^7$ points. For both simulations the pCN jump parameter $\beta$ is taken to be 0.005, with choice of prior
\[
\mu_0 = N(0,(35^2-\mathcal{A}_N)^{-5}).
\]

\subsection{Results}
\label{ssec:results}
\subsubsection{Recovery}
Figures \ref{fig:output_A} and \ref{fig:output_B} show conductivities arising from the MCMC chains, and Figure \ref{fig:phivalues} shows the values of the misfit $\Phi$ at the different sample means. \mmd{We consider two different types of mean conductivities. First the sample means are calculated in the sample spaces $X_i$ and then pushed forward to the conductivity space by the maps $F_i$, so that we produce estimates of $F_i(\mathbb{E}(u))$. This preserves the binary nature of the fields in the cases of the star-shaped and level set priors. We also consider the sample means of the pushforwards of the posteriors by $F_i$, $\mathbb{E}(F_i(u))$, which do not preserve the binary property but provide some visualisation of the interface uncertainty in the posterior.}

For Conductivity A, the sample mean $F_2(\mathbb{E}(u))$ arising from the star-shaped prior provides a better reconstruction than the other two prior choices. This is expected, since the true conductivity was drawn from this prior. Whilst the sample means arising from the level set prior are fairly close to the true conductivity (both visually and in terms of $\Phi$), the boundary of the interface appears to have too large a length-scale. Appropriate choice of prior length-scale is a key issue the with the level set method; treating the length-scale hierarchically as another unknown in the problem may be beneficial. On the other hand

The sample means arising from the Gaussian prior fail to recover both the sharp interface and the values of the conductivity, which is reflected in the large values $\Phi$ takes. \mmd{Nonetheless, general qualitative properties of the true conductivity can be seen in the sample means. As with the level set prior, reconstruction could be improved by treating hierarchically the parameters in the prior such as length-scale, marginal variance and mean.}

For Conductivity B, the level set prior is most effective in the reconstruction, since a specific number of inclusions isn't fixed a priori as it is for the star-shaped prior. Again the Gaussian prior fails to recover both the sharp interface and the values of the conductivity, however it appears to do a better job than the star-shaped prior at identifying the location and shape of the two inclusions.

In both of the above cases, even though the individual samples coming from using the level set prior contain many small inclusions, these do not show up in the \mmd{sample means $F_3(\mathbb{E}(u))$, though they are visible in the means $\mathbb{E}(F_3(u))$}. 

\mmd{In terms of the values of the misfit as shown in Figure \ref{fig:phivalues}, in all but one of the above cases the mean $\mathbb{E}(F_i(u))$ provides a better fit than $F_i(\mathbb{E}(u))$ despite the lack of the binary property. The quality of the recovery cannot be assessed solely on the misfit however; an advantage of the Bayesian approach is that both means are available (along with the full posterior), and whichever is more appropriate could depend on the context.}

\subsubsection{Convergence}
In Figure \ref{fig:ess}, we show the approximate effective sample size \mmd{(ESS)} associated with different quantities of interest. For all choices of prior, these are significantly smaller than the total $2.5\times 10^{6}$ samples generated. Many more samples may hence be required to produce accurate approximations of the posterior mean.

The chain associated with the star-shaped prior results in the largest ESS, likely because we are only attempting to infer $2^8 + 2$ parameters rather than $2^{14}$ parameters as in the log-Gaussian and level set cases.

In order to accelerate the convergence of the MCMC we can adjust the jump parameters $\beta$ and $\delta$. Larger choices of these parameters mean that accepted states will be less correlated with the current state, however the proposed states are less likely to be accepted. The choice $\beta = 1$ in pCN produces proposed states that are independent of the current state, but  dependent upon how far the prior is from the posterior, very few or no states may be accepted so that the chain never moves. Similarly, smaller choices of these jump parameters mean that more proposals will be accepted, but the states will be more correlated. A balance hence must be achieved -- in our simulations we choose the parameters such that approximately 20-30\% of proposals are accepted, though in general the optimal acceptance rate is not known \cite{acceptance}.

Alternatively, reconstruction may be accelerated by looking at an approximation of the posterior instead of the exact posterior, for example using the ensemble Kalman filter \cite{enkf} or a sequential Monte Carlo method \cite{smc}. We could also initialise the MCMC chains from EnKF estimates to significantly reduce the burn-in period and hence computational cost. If the derivative of the forward map is available, Hybrid Monte Carlo (HMC) methods could be used to accelerate the convergence \cite{hmc}. Emulators could also be used to reduce the computational burden of derivative calculation, allowing the use of geometric MCMC methods such as Riemannian Manifold Hamiltonian Monte Carlo (RHMC) and Lagrangian Monte Carlo (LMC) \cite{emulators}.

\subsubsection{Posterior Behaviour}
In Figures \ref{fig:d_gauss}-\ref{fig:d_lvl} we show kernel density estimates for a number of quantities associated with each posterior distribution. \mmd{These are calculated using subsequences of the MCMC chains, with lengths of the same order of magnitude as the effective sample sizes, in order to avoid over-fitting.} The most regular densities arise in the star-shaped case, with the distribution of all quantities appearing to be very close to uni-modal. More irregularity is seen for the log-Gaussian case, especially in the joint distributions, but they are still very close to uni-modal.

The least regular,  more multi-modal densities come from the level set prior. One reason for this is likely the lack of identifiability of the level set function: the forward model only `sees' the zero level set of the state, and hence cannot distinguish between infinitely many different states. The prior can however distinguish between these states, and will weight them appropriately, which can help explain the shape of the posterior densities.

\begin{figure}
\begin{subfigure}{\textwidth}
\centering
\includegraphics[width=0.3\textwidth]{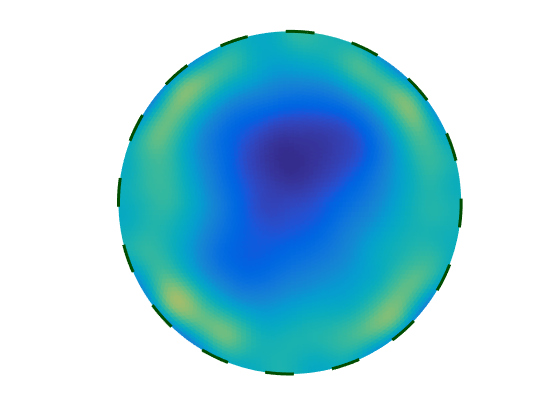}
\includegraphics[width=0.3\textwidth]{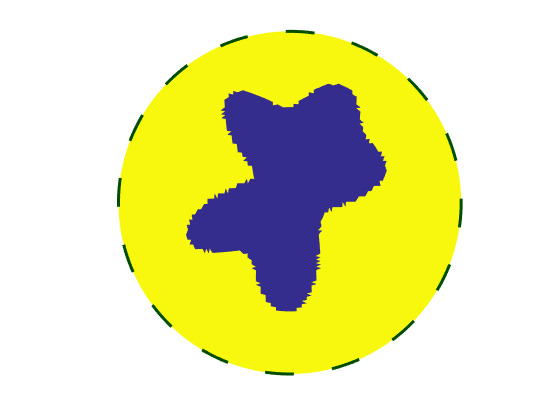}
\includegraphics[width=0.3\textwidth]{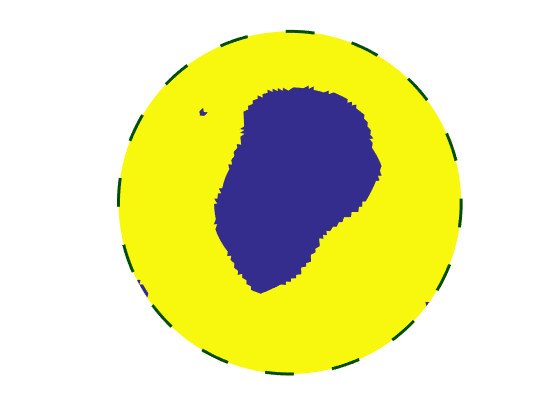}
\includegraphics[width=0.3\textwidth]{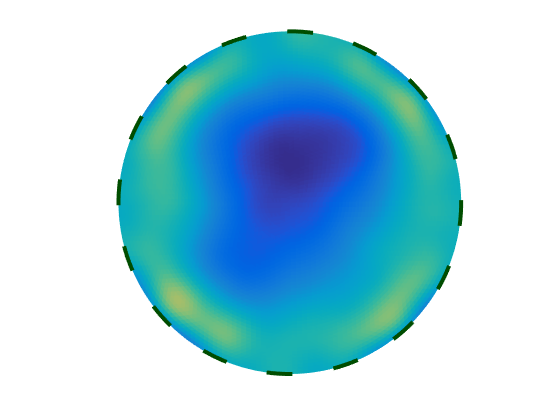}
\includegraphics[width=0.3\textwidth]{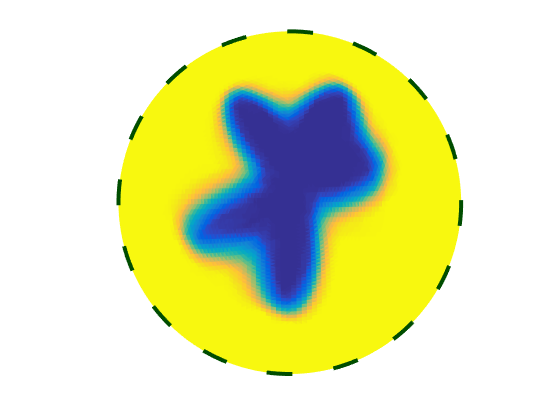}
\includegraphics[width=0.3\textwidth]{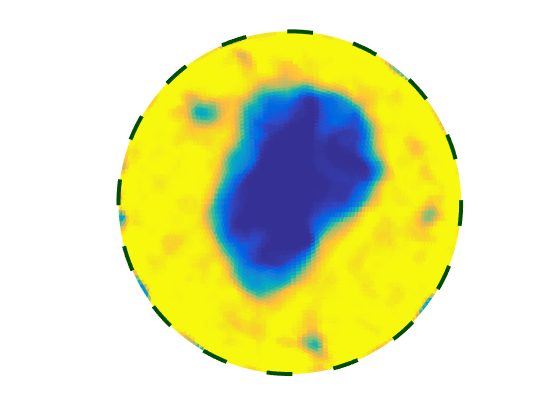}
\caption{(Top) Estimates of $F_i(\mathbb{E}(u))$. (Bottom) Estimates of $\mathbb{E}(F_i(u))$.}
\end{subfigure}\vspace{0.5cm}

\begin{subfigure}{\textwidth}
\centering
\includegraphics[width=0.3\textwidth]{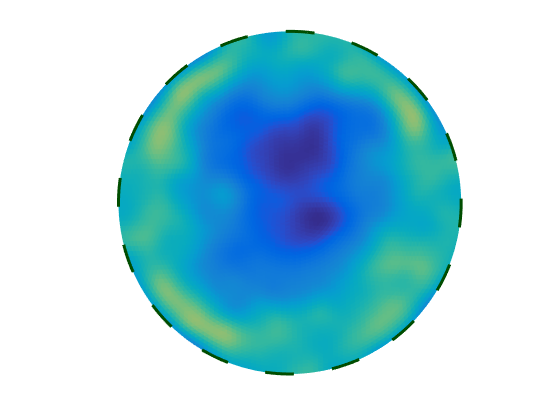}
\includegraphics[width=0.3\textwidth]{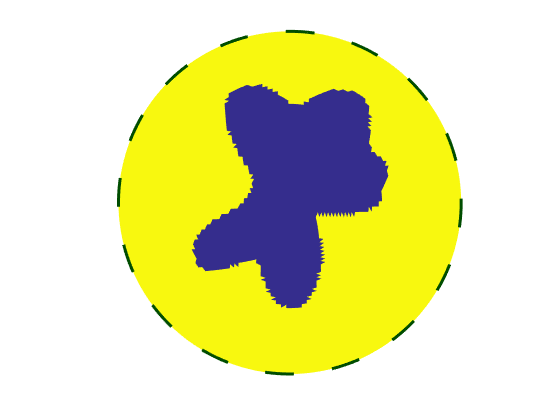}
\includegraphics[width=0.3\textwidth]{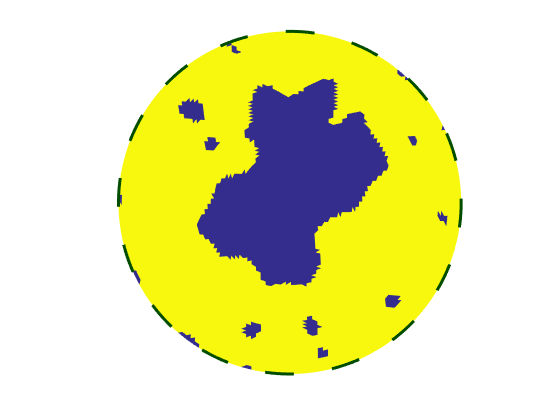}
\includegraphics[width=0.3\textwidth]{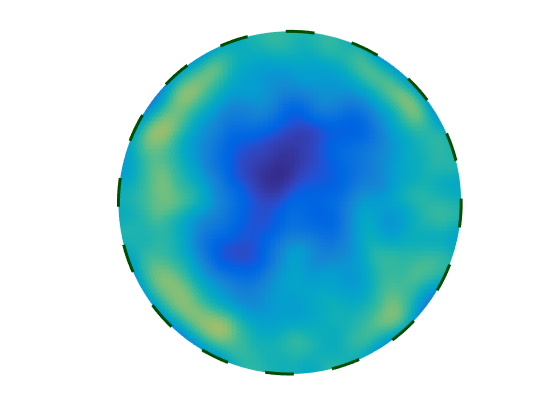}
\includegraphics[width=0.3\textwidth]{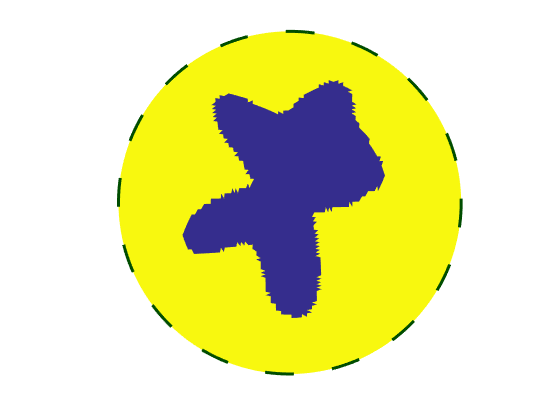}
\includegraphics[width=0.3\textwidth]{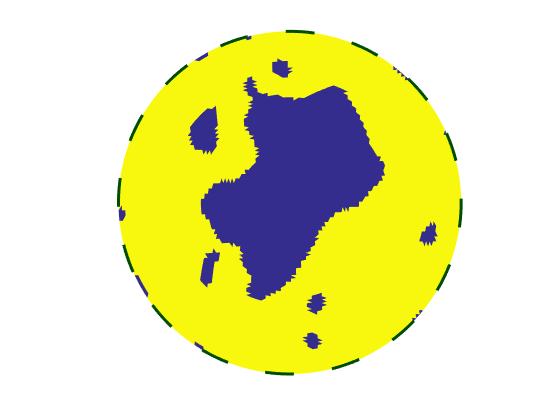}
\includegraphics[width=0.3\textwidth]{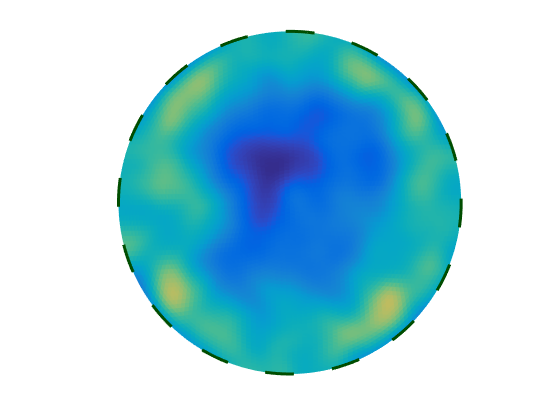}
\includegraphics[width=0.3\textwidth]{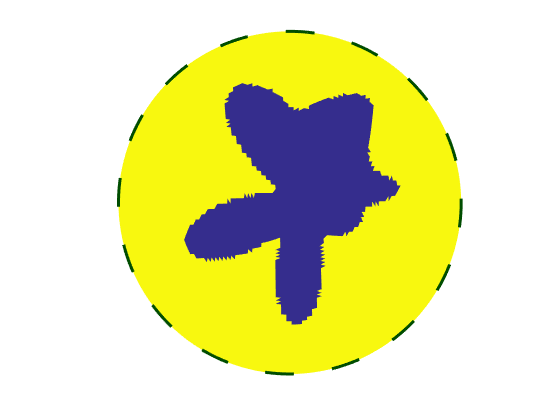}
\includegraphics[width=0.3\textwidth]{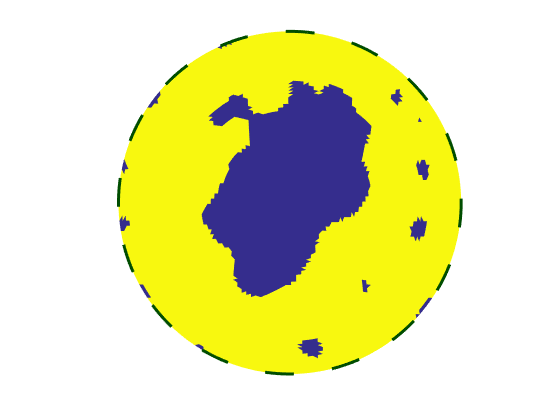}
\includegraphics[width=0.3\textwidth]{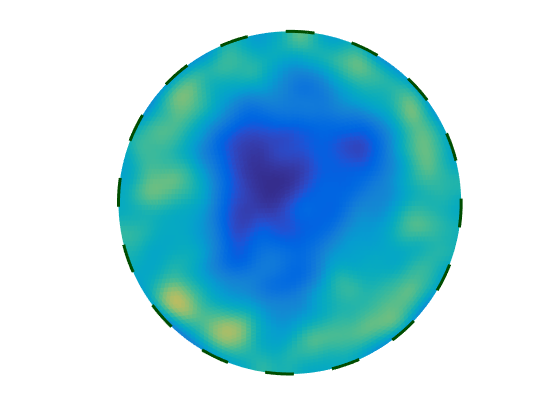}
\includegraphics[width=0.3\textwidth]{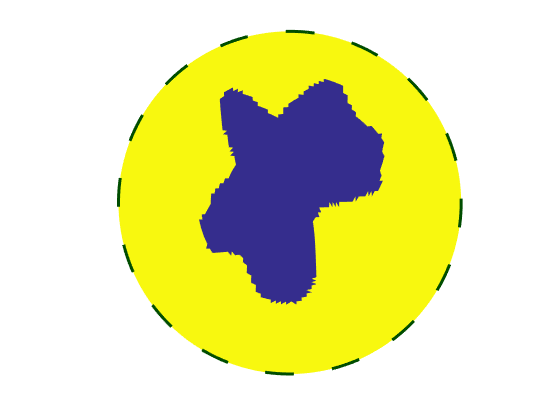}
\includegraphics[width=0.3\textwidth]{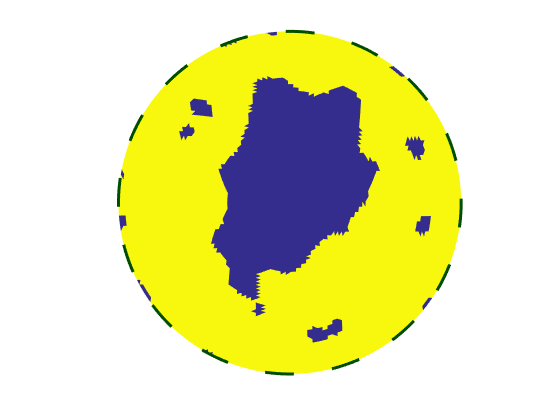}
\caption{Posterior samples.}
\end{subfigure}
\caption{Recovery of Conductivity A. From left to right, the log-Gaussian, star-shaped and level set priors are used.}
\label{fig:output_A}
\end{figure}

\begin{figure}
\begin{subfigure}{\textwidth}
\centering
\includegraphics[width=0.3\textwidth]{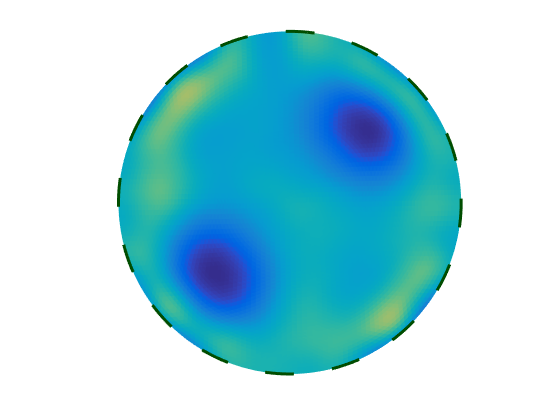}
\includegraphics[width=0.3\textwidth]{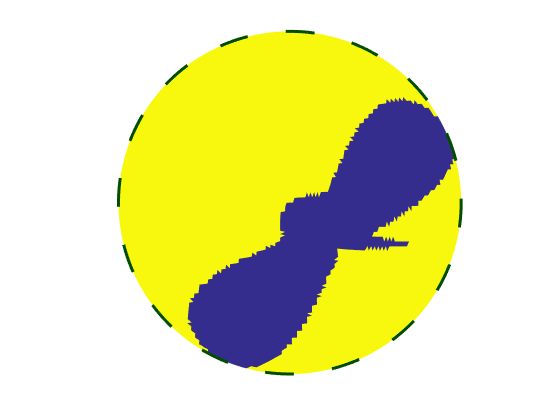}
\includegraphics[width=0.3\textwidth]{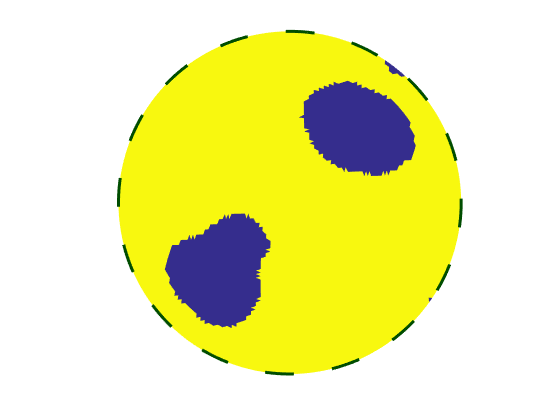}
\includegraphics[width=0.3\textwidth]{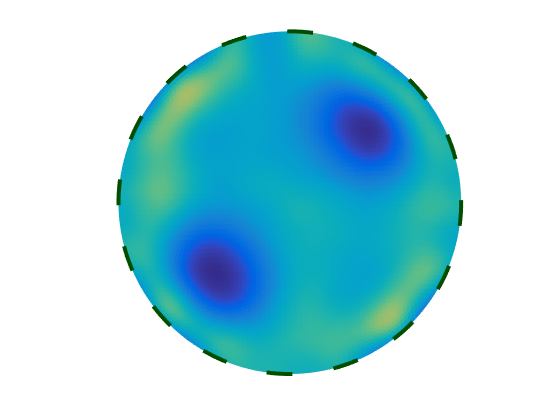}
\includegraphics[width=0.3\textwidth]{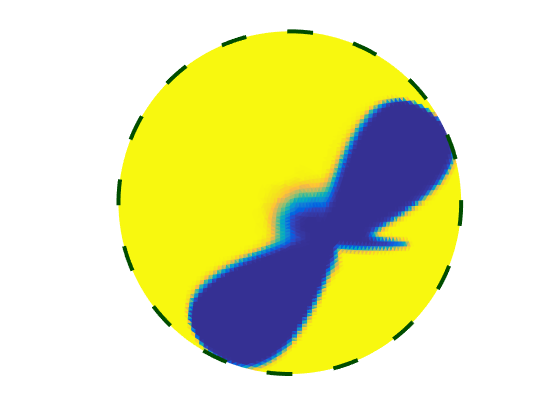}
\includegraphics[width=0.3\textwidth]{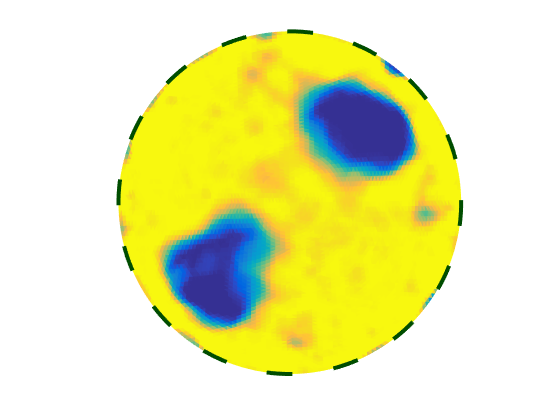}
\caption{(Top) Estimates of $F_i(\mathbb{E}(u))$. (Bottom) Estimates of $\mathbb{E}(F_i(u))$.}
\end{subfigure}\vspace{0.5cm}

\begin{subfigure}{\textwidth}
\centering
\includegraphics[width=0.3\textwidth]{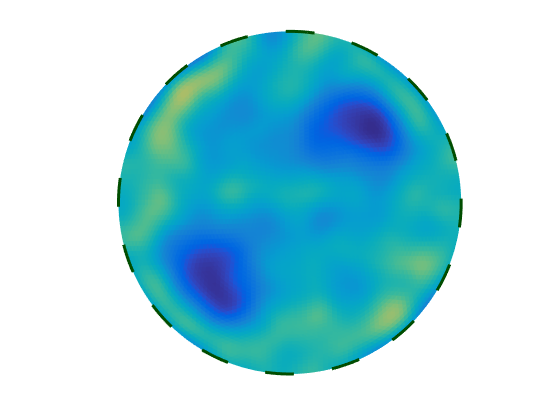}
\includegraphics[width=0.3\textwidth]{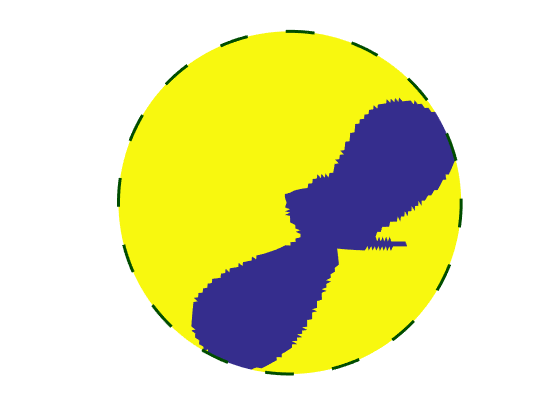}
\includegraphics[width=0.3\textwidth]{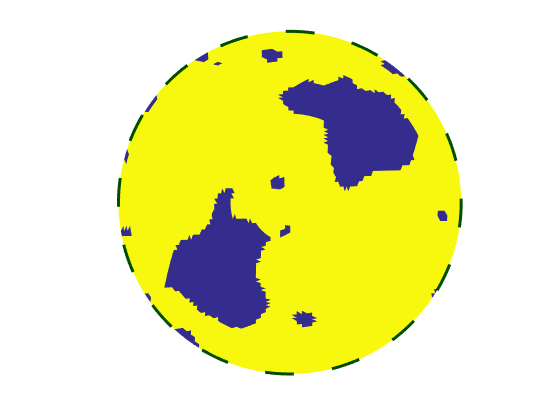}
\includegraphics[width=0.3\textwidth]{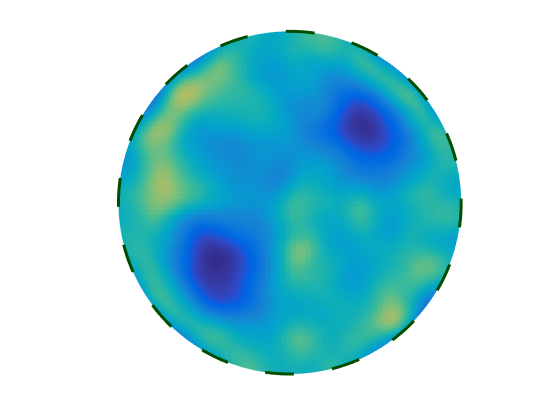}
\includegraphics[width=0.3\textwidth]{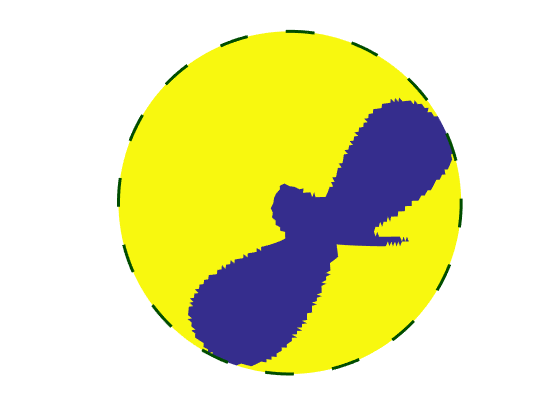}
\includegraphics[width=0.3\textwidth]{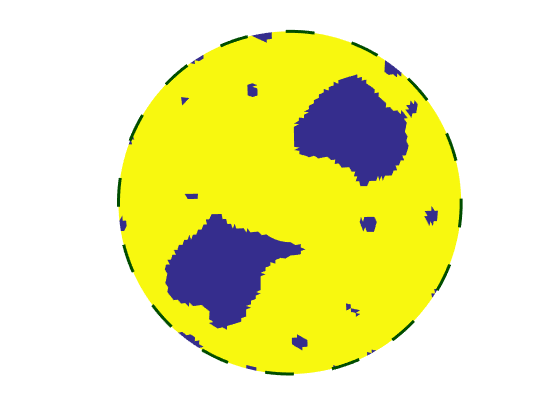}
\includegraphics[width=0.3\textwidth]{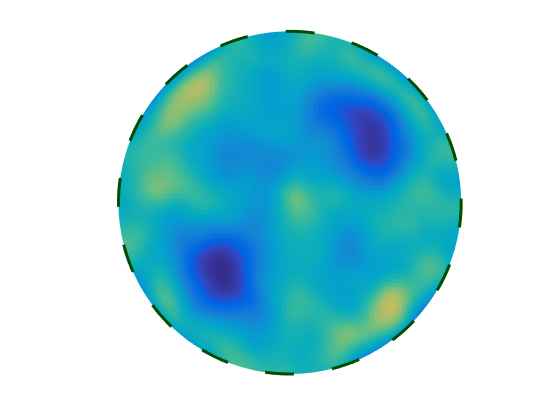}
\includegraphics[width=0.3\textwidth]{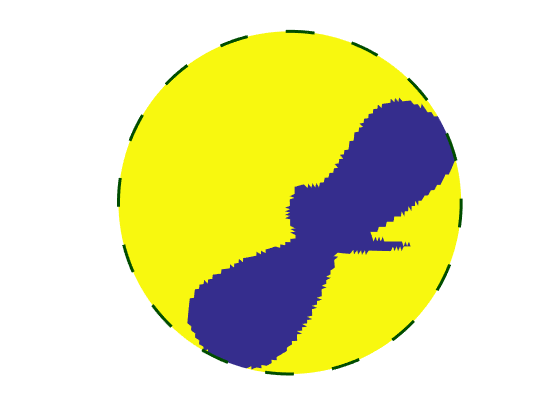}
\includegraphics[width=0.3\textwidth]{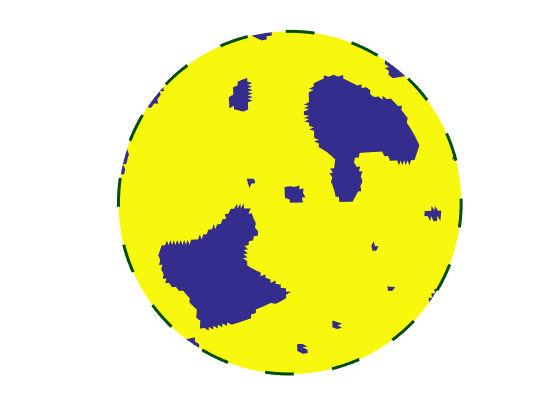}
\includegraphics[width=0.3\textwidth]{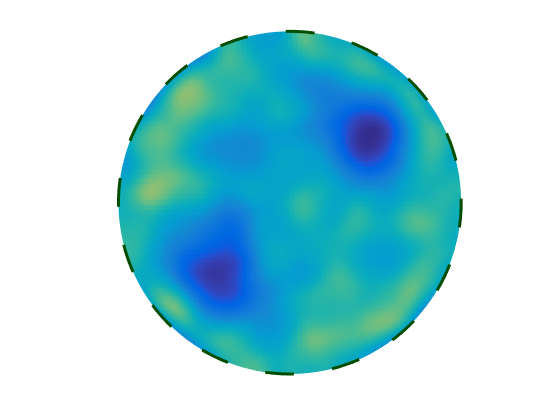}
\includegraphics[width=0.3\textwidth]{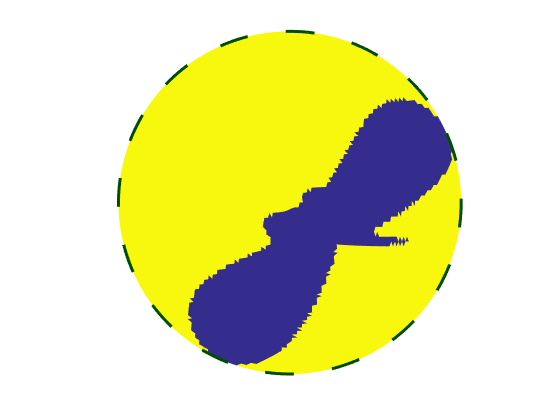}
\includegraphics[width=0.3\textwidth]{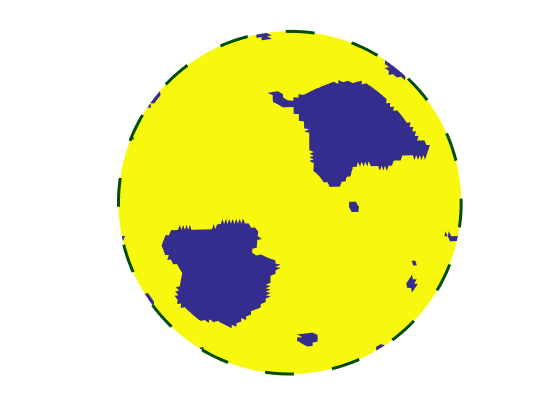}
\caption{Posterior samples.}
\end{subfigure}
\caption{Recovery of Conductivity B. From left to right, the log-Gaussian, star-shaped and level set priors are used.}
\label{fig:output_B}
\end{figure}

\begin{figure}
\begin{subfigure}{\textwidth}
\centering
\def\arraystretch{1.2}
\begin{tabular}{c|c|c}
& Conductivity A & Conductivity B\\
\hline
Log-Gaussian prior & 115960 & 122180\\
Star-shaped prior & 228.84 & 45013\\
Level set prior & 284.28 & 193.03
\end{tabular}
\caption{The values of $\Phi$ at $F_i(\mathbb{E}(u))$.}
\end{subfigure}\vspace{0.5cm}
\begin{subfigure}{\textwidth}
\centering
\def\arraystretch{1.2}
\begin{tabular}{c|c|c}
& Conductivity A & Conductivity B\\
\hline
Log-Gaussian prior & 111960 & 117950\\
Star-shaped prior & 284.44 & 41408\\
Level set prior & 137.93 & 124.42
\end{tabular}
\caption{The values of $\Phi$ at $\mathbb{E}(F_i(u)))$.}
\end{subfigure}
\caption{The values of the misfit $\Phi$ at the different mean conductivities, for the different conductivities and prior distributions.}
\label{fig:phivalues}
\end{figure}

\begin{figure}
\begin{center}
\begin{subfigure}{.45\textwidth}
\begin{center}
\def\arraystretch{1.2}
\begin{tabular}{c|c}
Quantity & Estimated ESS\\
\hline
$\hat{u}(0,1)$ & 40.0\\
$\hat{u}(0,2)$ & 90.7\\
$\hat{u}(1,1)$ & 35.4\\
$\hat{u}(1,2)$ & 44.5\\
$\hat{u}(1,3)$ & 36.0\\
$\hat{u}(2,1)$ & 101.9\\
$\hat{u}(2,2)$ & 37.9\\
$\hat{u}(2,3)$ & 89.7
\end{tabular}
\caption{Log-Gaussian prior}
\end{center}
\end{subfigure}
\begin{subfigure}{.45\textwidth}
\begin{center}
\def\arraystretch{1.2}
\begin{tabular}{c|c}
Quantity & Estimated ESS\\
\hline
$x_0^{(1)}$ & 241.7\\
$x_0^{(2)}$ & 89.6\\
$\hat{r}(1)$ & 101.1\\
$\hat{r}(2)$ & 179.4\\
$\hat{r}(3)$ & 277.8\\
$\hat{r}(4)$ & 214.8\\
$\hat{r}(5)$ & 146.7\\
$\hat{r}(6)$ & 146.7
\end{tabular}
\caption{Star-shaped prior}
\end{center}
\end{subfigure}\vspace{0.5cm}

\begin{subfigure}{.45\textwidth}
\begin{center}
\def\arraystretch{1.2}
\begin{tabular}{c|c}
Quantity & Estimated ESS\\
\hline
$\hat{u}(0,1)$ & 26.4\\
$\hat{u}(0,2)$ & 28.9\\
$\hat{u}(1,1)$ & 27.2\\
$\hat{u}(1,2)$ & 23.5\\
$\hat{u}(1,3)$ & 23.5\\
$\hat{u}(2,1)$ & 26.0\\
$\hat{u}(2,2)$ & 26.6\\
$\hat{u}(2,3)$ & 24.1
\end{tabular}
\caption{Level set prior}
\end{center}
\end{subfigure}
\end{center}
\caption{(Conductivity B) The estimated effective sample size for each chain, approximated using a variety of quantities, for the different choices of prior. In all cases $2.5\times 10^6$ total MCMC samples are produced, with the initial $5\times 10^5$ discarded as burn-in.}
\label{fig:ess}
\end{figure}

\begin{figure}
\begin{center}
\includegraphics[trim = 2cm 1cm 2cm 0cm,width=\textwidth]{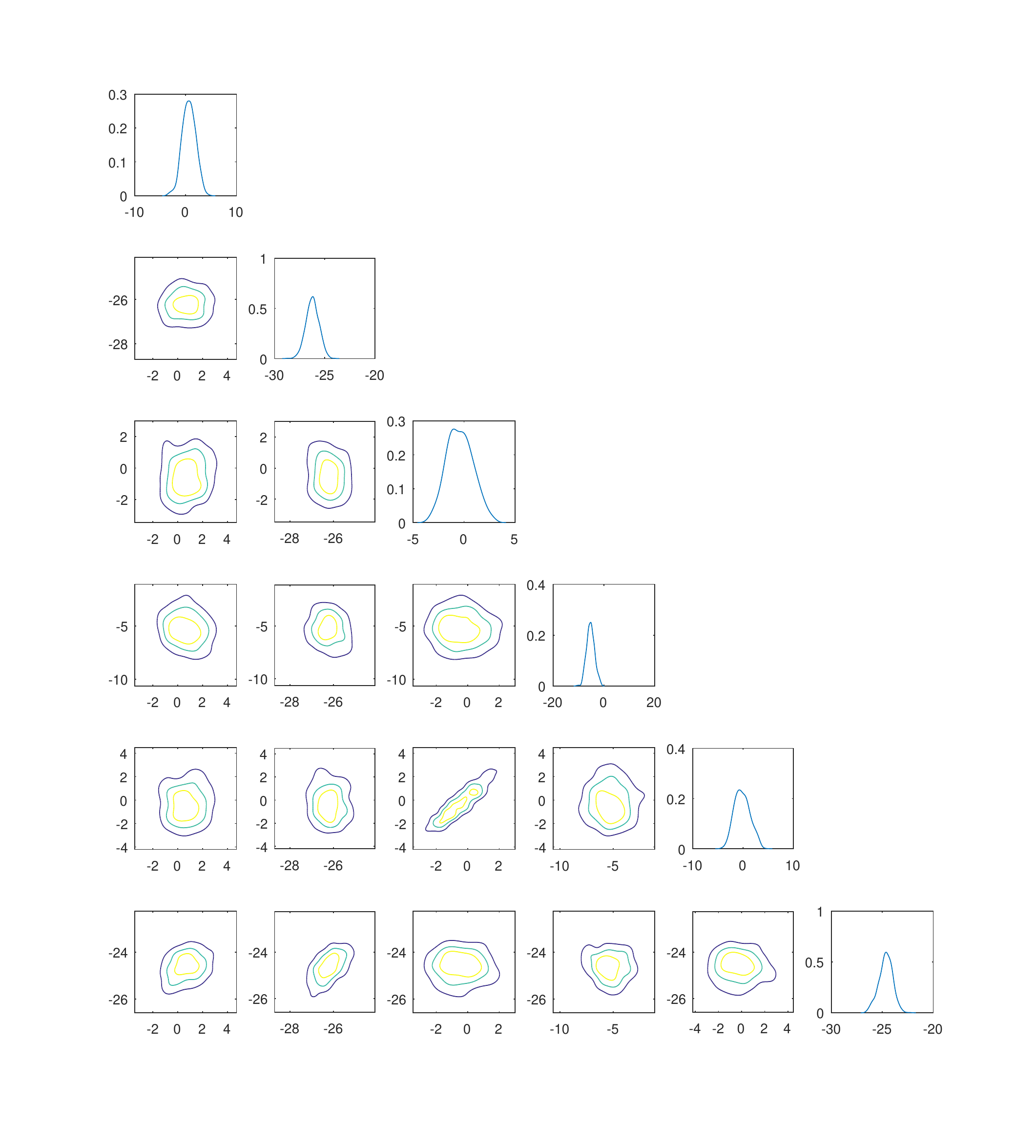}
\caption{(Conductivity B, log-Gaussian prior) Kernel density estimates associated with six Fourier coefficients of $u$. The diagonal displays the marginal densities of each coefficient, and the off-diagonals the marginal densities of corresponding pairs of coefficients.}
\label{fig:d_gauss}
\end{center}
\end{figure}

\begin{figure}
\begin{center}
\includegraphics[trim = 2cm 1cm 2cm 0cm,width=\textwidth]{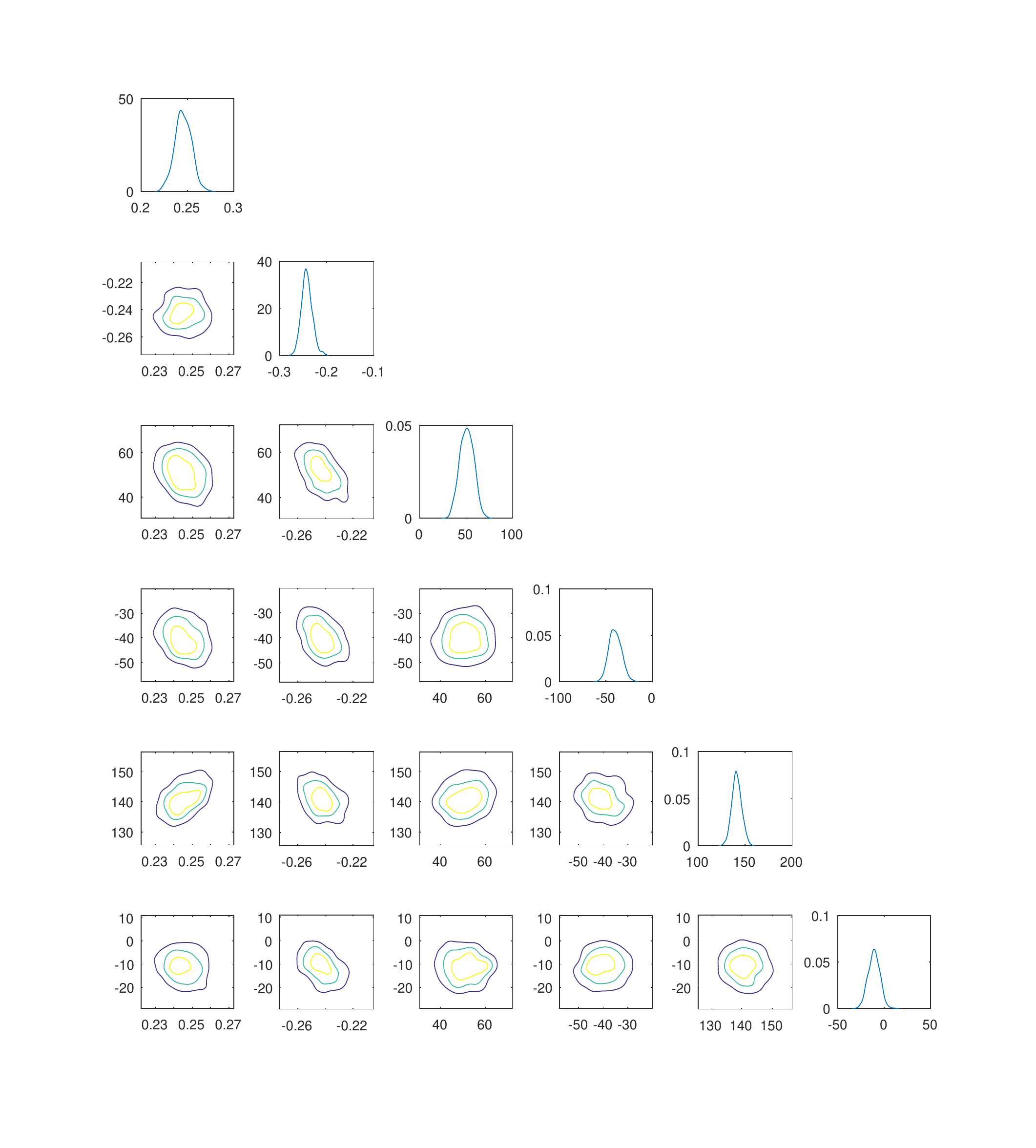}
\caption{(Conductivity B, star-shaped prior) Kernel density estimates associated with the centre $(x_0^{(1)},x_0^{(2)})$ and four Fourier coefficients of $r$. The diagonal displays the marginal densities of each quantity, and the off-diagonals the marginal densities of corresponding pairs of quantities.}
\label{fig:d_star}
\end{center}
\end{figure}

\begin{figure}
\begin{center}
\includegraphics[trim = 2cm 1cm 2cm 0cm,width=\textwidth]{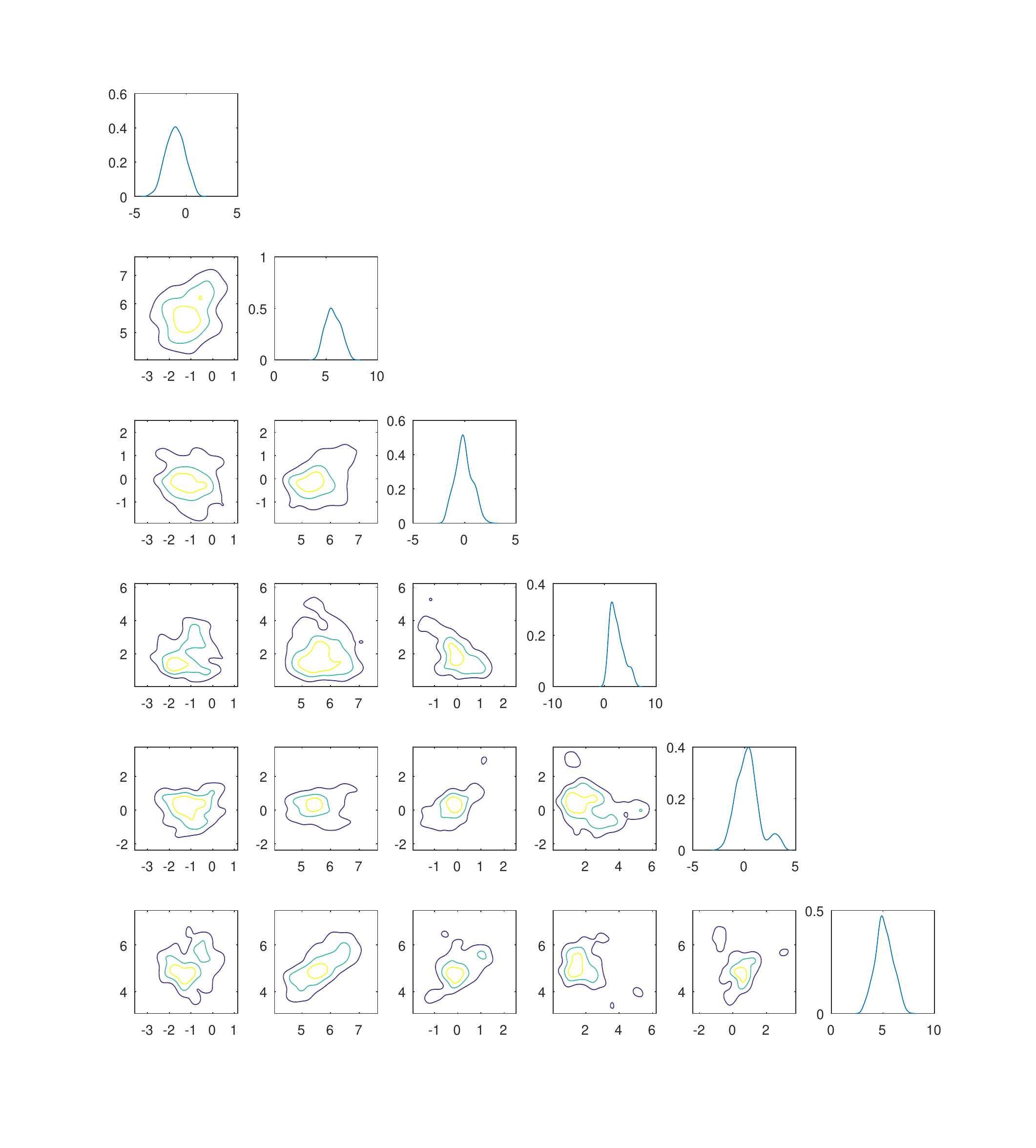}
\caption{(Conductivity B, level set prior) Kernel density estimates associated with six Fourier coefficients of $u$. The diagonal displays the marginal density of each coefficient, and the off-diagonals the marginal densities of corresponding pairs of coefficients. Axes are rescaled by $10^6$ for clarity.}
\label{fig:d_lvl}
\end{center}
\end{figure}

\section{Conclusions}
\label{sec:conc}
The primary contributions of this paper are:
\begin{itemize}
\item We have formulated the EIT problem rigorously in the infinite dimensional Bayesian framework.
\item We have studied three different prior models, each with their own advantages and disadvantages based on prior knowledge and the nature of the field we are trying to recover.
\item With each of these choices of prior we obtain well-posedness of the problem. We can obtain well-posedness using additional prior models, as long as the mapping from the state space to the conductivity space has appropriate regularity.
\item The infinite dimensional formulation of the problem leads to the use of state of the art function space MCMC methods for sampling the posterior distribution.
\item Simulations performed using these methods illustrate that the conditional mean provides a reasonable reconstruction of the conductivity, even with fairly significant noise on the measurements. They also illustrate the fact that the choice of prior has a significant impact on reconstruction and, in particular, that the geometric priors (star-shaped and level set)
can be particularly effective for the (approximately) piecewise constant fields that arise in many applications.
\end{itemize}
Future research directions could include the following:
\begin{itemize} 
\item Sampling the exact posterior distribution using MCMC can be computationally expensive. Methods that approximate the posterior may be as effective for calculating quantities such as the conditional mean, with much lower computational load. The relative effectiveness versus cost of different methods could be studied. This could be especially important for simulations in three dimensions, where forward model evaluations are even more expensive.
\item When using the level set prior, the length scale of samples could be treated hierarchically as an additional unknown in the problem.
\item The star-shaped prior could be extended to describe multiple inclusions, either with the number of inclusions fixed or as an additional unknown.
\mmd{\item The Calder\'on problem could be considered in place of the CEM. The data could then be either the full Neumann-to-Dirichlet (or Dirichlet-to-Neumann) map, or a finite number of pairs of Dirichlet and Neumann boundary values. In order to perturb the measurements with noise, the former case would require the definition of Gaussian distributions on spaces of operators, and the latter Gaussian distributions on manifolds.}
\end{itemize}

\section*{Acknowledgments} 
MMD is supported by EPSRC grant EP/HO23364/1 as part of the MASDOC DTC at the University of Warwick. AMS is supported by EPSRC and ONR. This research utilised Queen Mary's MidPlus computational facilities, supported by QMUL Research-IT and funded by EPSRC grant EP/K000128/1.

\providecommand{\href}[2]{#2}
\providecommand{\arxiv}[1]{\href{http://arxiv.org/abs/#1}{arXiv:#1}}
\providecommand{\url}[1]{\texttt{#1}}
\providecommand{\urlprefix}{URL }

\medskip

\end{document}